\documentclass[11pt,a4paper,twoside]{article}

\usepackage[utf8]{inputenc}
\usepackage[T1]{fontenc}
\usepackage{lmodern}
\usepackage{euler} 
\usepackage{enumitem}
\usepackage[english]{babel} \usepackage{amsmath,amssymb,amsthm} \usepackage{mathtools} \usepackage{comment}
\usepackage[final]{microtype}
\usepackage{changepage}

\usepackage{color}

\usepackage{soul}

\usepackage{geometry}
\usepackage{titlesec}
\usepackage{fancyhdr}
\newcommand{\shorttitle}{Finite semisimplicial sets, differential graded algebras, and cellular sheaves}

\titleformat{\section}[block]
  {\centering\scshape\small} 
  {\thesection}{0.6em}{}    

\titleformat{\subsection}[block]
  {\centering\scshape\footnotesize}
  {\thesubsection}{0.5em}{}

\titlespacing*{\section}{0pt}{2.2ex plus .6ex}{1.2ex}
\titlespacing*{\subsection}{0pt}{1.6ex plus .5ex}{0.8ex}
\titlespacing*{\subsubsection}{0pt}{1.2ex plus .4ex}{0.6ex}

\usepackage{hyperref}
\usepackage{cleveref}
\fancypagestyle{bibpages}{%
  \fancyhf{}%
  \fancyhead[LO,RE]{\small\thepage}%
}


\hypersetup{
colorlinks=true,
linkcolor=black,
citecolor=black,
urlcolor=black
}

\usepackage{tikz}
\usepackage{tikz-cd}

\theoremstyle{plain}
\newtheorem{theorem}{Theorem}[section]
\newtheorem{lemma}[theorem]{Lemma}
\newtheorem{proposition}[theorem]{Proposition}
\newtheorem{corollary}[theorem]{Corollary}
\newtheorem{mainthm}{Theorem}

\theoremstyle{definition}
\newtheorem{definition}[theorem]{Definition}
\newtheorem{example}[theorem]{Example}
\newtheorem{remark}[theorem]{Remark}
\newtheorem{excerpt}[theorem]{}

\newcommand{\cC}{\mathcal{C}}
\newcommand{\cA}{\mathcal{A}}
\newcommand{\cB}{\mathcal{B}}

\newcommand{\cL}{\mathcal{L}}

\newcommand{\Sets}{\mathrm{Sets}}

\newcommand{\ssreg}{\mathsf{SS}^{\mathrm{ns}}}
\newcommand{\ssregh}{\widehat{\mathsf{SS}}^{\mathrm{ns}}}
\newcommand{\digraphs}{\mathsf{DiGraph}_{\leq 1}}

\renewenvironment{proof}[1][Proof]{
  \par\noindent\textbf{#1.}\quad
}{
  \qed\par\addvspace{0.8\baselineskip}%
}
\newenvironment{proof*}[1][Proof]{%
  \par\noindent\textbf{#1.}\quad
}{
  \par\addvspace{0.8\baselineskip}%
}

\newlength{\TitleAuthorSep}
\newlength{\AuthorAbstractSep}
\newlength{\AbstractIntroSep}
\newlength{\AuthorAffilSep}
\newcommand{\papertitle}[3]{
  \begin{center}
    {\Large\scshape #1\par}
    \vspace{\TitleAuthorSep}
    {\normalsize\scshape #2\par}
    \vspace{\AuthorAffilSep}
    {\footnotesize #3\par}
  \end{center}
  \vspace{\AuthorAbstractSep}
}

\newenvironment{narrowabstract}{
  \par\vspace{-0.1\baselineskip}
  \begin{adjustwidth}{+12mm}{+12mm}
  \footnotesize\noindent\textsc{Abstract}. 
}{
  \end{adjustwidth}\par\vspace{1.2\baselineskip}
}

\renewenvironment{narrowabstract}{
  \par\vspace{-0.1\baselineskip}
  \begin{adjustwidth}{+12mm}{+12mm}
  \footnotesize\noindent\textsc{Abstract}. 
}{
  \end{adjustwidth}\par\vspace{\AbstractIntroSep}
}

\begin{document}

\papertitle
  {Finite semisimplicial sets, differential graded algebras, and cellular sheaves}
  {J.-W. van Looy\textsuperscript{1}
   \qquad
   F. Zanchetta\textsuperscript{2.*}}
  {%
    \textsuperscript{1}University of Bologna,
    Bologna, Italy\\
    \href{mailto:janwillem.vanlooy@unibo.it}{\texttt{janwillem.vanlooy@unibo.it}}
    \\[0.7ex]
    \textsuperscript{2}University of Bologna,
    Bologna, Italy\\
    \href{mailto:ferdinando.zanchett2@unibo.it}{\texttt{ferdinando.zanchett2@unibo.it}}
    \\[0.7ex]
    \textsuperscript{*}Corresponding author
  }
\begin{narrowabstract} \noindent We develop a differential graded model for finite non-singular semisimplicial sets and their cellular sheaves. To a finite non-singular semisimplicial set \(S\), we associate an exterior DGA \(\Omega^\bullet_S\), extending the classical anti-equivalence between finite simple directed graphs and first-order differential calculi to higher degrees. We characterize the essential image of this construction, obtaining an equivalence of categories that recovers the graph-FODC correspondence in dimension one. For a fixed \(S\), we then characterize a category of differential graded \(\Omega^\bullet_S\)-modules equivalent to the category of cellular sheaves on $S$, thereby giving a differential refinement of the usual incidence-algebra description. Finally, we study connections and curvature in this framework. Connections are described by edgewise linear maps and their curvature on a 2-dimensional simplex is the difference between direct and composite edge transport. When the edge transports are invertible, vanishing of this curvature on every \(2\)-simplex can equivalently be seen as a gluing criterion to extend the given edge transports to a connection sheaf on \(P_S\).
\end{narrowabstract}

\thispagestyle{empty} 
\section{Introduction}

A cellular sheaf \cite{Ayzenberg2025SheafTheory, curry2014SheavesCosheavesAndApplications} on a finite non-singular complex \( S\) can be described as a covariant functor from its face poset \( P_S \) to the category of finite-dimensional vector spaces over \(\Bbbk\), which we denote by \( \mathrm{Vect}_\Bbbk \). The purpose of this paper is to give this combinatorial datum a differential graded model. To a finite non-singular semisimplicial set $ S\in \ssreg_{\leq n}$ of dimension $n$ (we also require that the simplices of our spaces are uniquely determined by their vertices, with precise definitions given in Definition \ref{definition:non-singular-complexes}), we associate an exterior differential graded algebra $ \Omega_S^\bullet $, and we characterize the DGAs that arise in this way. For fixed $S$, we then characterize the DG-modules over $ \Omega_S^\bullet$ that come from cellular sheaves on $ S $. In doing so, we set up a dictionary between semisimplicial combinatorics, cellular sheaves, and differential graded algebra.

The first correspondence is guided by the classical correspondence between directed graphs and first-order differential calculi (FODC) on algebras of functions on finite sets \cite{DimakisMullerHoissen1994Discrete, Majid2012Graph}, which we recall in Example \ref{example: graph differential calculus}. A non-singular \( 1 \)-dimensional semisimplicial set has no loops among its edges. Its free degeneracy completion, usually required to establish a correspondence between graphs and FODC, adds a degenerate loop at each vertex, thereby allowing edge-collapsing morphisms, while the differential forms discard these loops. The same mechanism extends to higher degrees. For \( S \in \ssreg_{\leq n}\), write \( \widehat S := UL(S) \), where \( L(S) \) is the free degeneracy completion of \(S\) (formally we look at it as a simplicial set) and $U$ is the forgetful functor associating to a simplicial set its underlying semisimplicial set. Denote by \( \ssregh_{\leq n} \) the category of semisimplicial sets obtained with this procedure. We think of \( \widehat{S} \) as the higher dimensional analogue of the free degeneracy completion of a non-singular graph. Degenerate simplices occur in the simplicial completion as they are required to allow simplex-collapsing maps, but they do not represent additional cells and contribute no basis forms after normalization.

A natural candidate for the higher-degree calculus is the maximal prolongation of the first-order calculus on the \(1\)-skeleton. An explicit description of the maximal prolongation is given in \cite{BeggsMajid2020QRG}, and a complementary description of local degree-two calculi appears in \cite[\S 5]{BrzezinskiMajid1997}. The maximal prolongation, however, is determined entirely by the first-order calculus and hence, by the \(1\)-truncation of \(S\). As such, it cannot distinguish semisimplicial sets with isomorphic \( 1\)-truncations but different higher-dimensional simplices. We instead define $$ \Omega^\bullet_S := N^\bullet(L(S); \Bbbk),$$ where
\( N^\bullet ( L(S) ; \Bbbk) \) denotes the normalized cochain complex equipped with its standard cup product \cite[\S 3.2]{Hatcher_2002}. The use of normalization naturally recalls the Dold-Kan correspondence \cite{weibelhom}, but the result we present below is of a different nature: Dold-Kan concerns simplicial objects in an abelian category and their underlying chain complexes, whereas here the multiplicative and support structure of the normalized cochains is essential for recovering the original semisimplicial set.
Since the nondegenerate simplices of \( L(S) \) are precisely the simplices of S, \(\Omega_S^p\) has a basis indexed by \( S_p \). As discussed in Remark \ref{remark: relation-with-max-prol}, there is a surjective DGA morphism \[ \Omega_{\max}^\bullet( \Omega_S^1) \twoheadrightarrow\Omega_S^\bullet,\] sending a path monomial of the maximal prolongation to the corresponding simplex form when its ordered vertex tuple is realized by a simplex of \(S\), and to zero otherwise. Thus the quotient retains the higher dimensional face data that is invisible to the calculus on the \(1\)-skeleton. 

Our first result shows that this construction loses no combinatorial information. Writing $\mathsf{DG}_n$ for the category of DGAs concentrated in degree $0,\ldots,n$, we prove that the resulting functor $$ \bigl( \ssregh_{\leq n}\bigr)^{\mathrm{op}}\rightarrow \mathsf{DG}_n, \quad\widehat{S}\rightarrow\Omega^\bullet_S$$ is fully faithful, and conditions {(D1)--(D3)} in Definition \ref{def: concrete DGAs essential image} characterize its essential image, which we denote by $\widehat{\mathsf{DG}}_n$. Roughly, these conditions say that degree zero is an algebra of functions on a finite set, that higher forms are generated by edge forms, and that the nonzero path monomials form a basis. This can be seen as the first main result of this paper.
\begin{mainthm}\label{theorem: A - for the introduction}[\ref{theorem: essential image characterization}]
The functor \( \Omega^\bullet\colon \bigl(\ssregh_{\leq n}\bigr)^{\mathrm{op}}
\rightarrow \widehat{\mathsf{DG}}_n\) is an equivalence of categories. \end{mainthm} 

For \( n=1 \), this recovers the graph--calculus anti-equivalence. In higher degrees, it shows in particular that the differential graded algebra $\Omega^\bullet_S$, rather than only its first-order part, determines the semisimplicial structure. The second main correspondence we establish in this paper concerns cellular sheaves. By the standard equivalence between functors on a poset and sheaves on its Alexandrov space \cite[\S 4.2]{curry2014SheavesCosheavesAndApplications}, functors on the face poset \( P_S \) of a non-singular complex \( S \) may be identified with sheaves on the Alexandrov space \(X_{P_S}\). We denote this category $\mathsf{Sh}( S )$. For a given \( F \in \mathsf{Sh}(S)\), we can consider the graded vector space \[ M_F^\bullet := \bigoplus_{p=0}^n \bigoplus_{\sigma\in S_p} F( \sigma).\] The restriction maps of \( F\) determine an \( \Omega_S^\bullet \)-action together with a differential on \( M_F^\bullet \). Conversely, a simplexwise decomposition and suitable support condition on the action and differential allow the sheaf restriction maps to be recovered from a DG-module. As a consequence, in Section \ref{section: cellular sheaves, differential modules} we identify a subcategory of left DG-modules, that we call cellular DG-modules and denote by $\prescript{\mathsf{c}}{S}{\mathcal{M}}^{\mathsf{DG}}$, leading to the following result.
\begin{mainthm}\label{theorem: B - for the introduction}
There is an equivalence of categories \( \mathsf{Sh}(S)\simeq \prescript{\mathsf{c}}{S}{\mathcal{M}}^{\mathsf{DG}}\).\end{mainthm}

This result can be seen as differential refinement of the equivalence between cellular sheaves on \(P_S\) and finite-dimensional right modules over the incidence algebra \(\Bbbk[P_S]\) as discussed in \cite{Ladkani2008DerivedEquivalencesSheavesFinitePosets, Vybornov1999, Yanagawa2005Dualizing}. We also make precise the comparison with the incidence-algebra module at the end of Section \ref{section: cellular sheaves, differential modules}. The differential graded formulation has the additional advantage that differential constructions such as connections and curvature are internal to the same algebraic framework.

We develop this perspective in Sections \ref{section: connections and curvature} and \ref{section: connection sheaves and connections}. For a finite-dimensional left \(A_S\)-modules, we show that left connections relative to $\Omega_S^\bullet$ always exist and are naturally parametrized by families of linear maps $$R_e\colon M_{d_1(e)}\rightarrow M_{d_0(e)}, \ e\in S_1$$Their curvature can be computed explicitly in terms of these maps: on a 2-dimensional simplex $\sigma=[v_0,v_1,v_2]$ its coefficient is the difference $R^\sigma_{02} -R^\sigma_{12} R^\sigma_{01}$, measuring the difference between the direct transport from $v_0$ to $v_2$ and transport through $v_1$.

When the maps \( R_e \) are isomorphisms, the curvature of the connection can be related to a gluing criterion: prescribed vertex spaces and edge transports extend to a connection sheaf on $S$, that is a sheaf on $S$ whose restriction maps associated to the face incidence relations are isomorphisms, if and only if the curvature vanishes on every $2$-dimensional simplex. This is the content of Proposition \ref{proposition: edge transport extension criterion}. Under the standard correspondence between functors on a poset and sheaves on its Alexandrov space \cite[\S 4.2]{curry2014SheavesCosheavesAndApplications}, connection sheaves correspond precisely to locally constant sheaves on \(X_{P_S}\). Equivalently, they may be described as finite-dimensional quasi-coherent modules over the constant structure sheaf \( \underline{\Bbbk}_{X_{P_S}}\) \cite[Theorem~3.6]{Sancho2018HomotopyFiniteRingedSpaces}. If \(S\) is obtained from a finite ordered good cover of a locally path-connected space \(M\), sheaf descent relates these objects to \(\Bbbk\)-local systems on \(M\) \cite{Vistoli2005Descent}. For a smooth manifold and \(\Bbbk=\mathbb R\) or \(\mathbb C\), this leads to the usual correspondence with flat vector bundles \cite[\S9.2]{Voisin_I}. These correspondences are discussed at the end of section \ref{section: connection sheaves and connections}.

Our construction is also related to two recent works \cite{fioresi2026gluing,fioresi2026sheaves} of Fioresi, Simonetti, and Zanchetta. In \cite{fioresi2026sheaves}, cellular sheaves on graphs are related to discrete differential calculi, and the corresponding sheaf, transport, and noncommutative Laplacians are compared under suitable hypotheses. In \cite{fioresi2026gluing}, gluing data are encoded by finite ringed spaces associated with two-dimensional semisimplicial sets, with the aim of reconstructing schemes and with possible extensions to differentiable settings. This work also starts from the graph-calculus correspondence underlying the former, extends it to higher-dimensional semisimplicial sets, and uses the finite-space language of the latter to interpret connection sheaves, without addressing the corresponding reconstruction problem that would require giving to our cellular sheaf a richer structure.

Although the paper is purely theoretical and does not develop machine learning algorithms, part of its motivation comes from recent developments in geometric and topological deep learning. One recent direction extends graph-based constructions and algorithms to higher-dimensional combinatorial domains, including simplicial and regular cell complexes \cite{BodnarEtAl2021CWN, BodnarEtAl2021MPSN, EbliEtAl2020SimplicialNN,Papamarkou24, PapillonEtAl2023ArchitecturesTDL}; another enriches the coefficient system by replacing scalar data with cellular sheaves. The spectral theory of cellular sheaves and their Laplacians was developed by Hansen and Ghrist in \cite{HansenGhrist2019Spectral}, and subsequent work introduced sheaf neural networks and neural sheaf diffusion \cite{bodnar2022NeuralSheafDiffusion,HansenGebhart2020SheafNN, HansenGhrist2019LearningSheafLaplacians}, together with several extensions \cite{Battiloro2022TangentBundleFilters,borgi2025polynomial,Braithwaiteetal2024,HypergraphNeuralSheafDiffusionSimplicialSetframework,Duta2025SheafHypergraphNetworks,Zaghen2024SheafDiffusionNonlinear}. These two developments modify complementary parts of the graph setting: higher-order architectures enrich the underlying incidence structure, whereas sheaf-based architectures enrich its local coefficient system. Theorems A and B package these two ingredients as a differential graded algebra and a cellular DG-module over it. For a broader account of (cellular) sheaf theory in representation learning, we refer to \cite{Ayzenberg2025SheafTheory}.

The paper is organized as follows. Section \ref{section:preliminaries-and-conventions} fixes notation and recalls the required background on semisimplicial sets, cellular sheaves, and differential calculi. Section \ref{section: from semisimplicial sets to differential graded algebras} proves Theorem \ref{theorem: A - for the introduction}, while Section \ref{section: cellular sheaves, differential modules} proves Theorem \ref{theorem: B - for the introduction} and compares the construction with the incidence-algebra description. Section \ref{section: connections and curvature} develops connections and curvature, while Section \ref{section: connection sheaves and connections} discusses connection sheaves, the curvature obstruction to their extension from edge data, and the resulting geometric interpretations.

\begin{adjustwidth}{1.5em}{1.5em}
\noindent\textbf{Notation.} We record here some notation used throughout the paper for the convenience of the reader.
\end{adjustwidth}
\begin{adjustwidth}{1.5em}{2.0em}
\begin{itemize}[nosep]
\item We denote by \(\Delta\)
the simplex category. Its objects are the finite ordered sets
\([n]=\{0\to 1\to\ldots\to n\},\ n\in \mathbb N,\)
and its morphisms are the order-preserving maps. We denote by
\(\Delta^+\subset \Delta\) the subcategory whose morphisms are the
order-preserving injections. For \(n\in \mathbb N\), we write
\(\Delta_n\subset \Delta,
\ \Delta_{n}^+\subset \Delta^+\)
for the full subcategories on the objects \([0],\ldots,[n]\).
\item We write \(\mathsf{Sets}\) for the category of sets and \(\mathrm{Vect}_\Bbbk\) for the
category of finite-dimensional vector spaces over the base field \(\Bbbk\) for \(\Bbbk\) a field of characteristic different from \(2\).
\item If
\(\cC\) is a category, we denote its opposite category by
\(\cC^{\mathrm{op}}\). For categories \(\cA,\cB\), we write
\(\operatorname{Fun}(\cA,\cB)\)
for the category of covariant functors \(\cA\to\cB\), and
\(\operatorname{Pre}(\cA,\cB) :=
\operatorname{Fun}(\cA^{\mathrm{op}},\cB) \)
for the category of contravariant functors. We refer to objects of
\(\operatorname{Pre}(\cA,\cB)\) as \(\cB\)-valued presheaves
on \(\cA\). When \(\cB=\mathsf{Sets}\), or when the target category is
clear from context, we simply write \(\operatorname{Pre}(\cA)\).
\item A preorder on a set is a reflexive and transitive binary relation. A partial order is
an antisymmetric preorder. A set equipped with a partial order will be called a
poset.
\item For a topological space $X$ and a set $A$ (or abelian group, etc), we denote by $\underline{A}$ the constant sheaf on $X$ associated to $A$.
\end{itemize}
\end{adjustwidth}

\section*{Acknowledgements}
Both authors would like to thank Rita Fioresi and Angelica Simonetti for their support and for stimulating discussions about topics related to this paper. The first author gratefully acknowledges the hospitality of Jan Slovák, Katharina Neusser, and Erik Bekkers during research visits to their respective institutions undertaken while this work was in preparation. The second author would like to thank Thomas Weber and Francesco Vaccarino for stimulating discussions on topics related to this work. The first author was supported by MSCA-DN CaLiForNIA, Project ID: 101119552. The second author acknowledges support by GNSAGA-Indam, INFN Gast Initiative,
PNRR MNESYS, PNRR National Center
for HPC, Big Data and Quantum Computing CUP J33C22001170001, PNNR SIMQuSEC
CUP J13C22000680006. 
This work was also supported by Horizon Europe EU projects MSCA-SE CaLIGOLA,
Project ID: 101086123, COST Action CaLISTA CA21109.

\section{Preliminaries and conventions}\label{section:preliminaries-and-conventions}
In this section, we collect the background notions used throughout the paper. We first recall the semisimplicial language underlying our construction, together with the face-poset and Alexandrov descriptions of cellular sheaves. We then review first-order differential calculi,
their extension to exterior differential graded algebras, and the graph-calculus correspondence that motivates the higher-dimensional construction of \S\ref{section: from semisimplicial sets to differential graded algebras}.

For background on simplicial and semisimplicial sets, see, for instance, \cite{Friedman2021Simplicial, GoerssJardine1999}. Additional context for the free degeneracy completion can be found in \cite{EbertRandalWilliams2019SemiSimplicial}. For cellular sheaves on face posets, finite topological spaces, and the incidence-algebra perspective, we
refer to  \cite{Ayzenberg2025SheafTheory, Barmak_2011finitetopologicalspacesandapplications,curry2014SheavesCosheavesAndApplications, Ladkani2008DerivedEquivalencesSheavesFinitePosets} For differential calculi and differential graded algebras, see, for example, \cite{BeggsMajid2020QRG,Landi1997Introduction}.

\subsection{(Semi)simplicial sets, Alexandrov topologies, and cellular sheaves}\label{subsec_semisimplicial-sets}

\begin{definition}
The category of simplicial sets (resp. semisimplicial sets) is given by \( \mathsf{sSet} :=\operatorname{Pre}(\Delta, \mathsf{Sets}),\  \text{(resp. } \mathsf{ssSet}:=\operatorname{Pre}(\Delta^+,\mathsf{Sets}).\)) If \(X\) is a simplicial or semisimplicial set, we write
\( X_m:=X([m]) \) and refer to elements of \(X_m\) as the \emph{\(m\)-simplices of \(X\)}. We say that a semisimplicial set \(X\) has dimension \(n\) if \(X_n \neq \emptyset \) and \(X_m = \emptyset\) for all \( m > n \). We define \emph{\( n \)-truncated simplicial and \( n \)-truncated semisimplicial sets} by \(\mathsf{sSet}_{\leq n}:=\operatorname{Pre}(\Delta_n,\mathsf{Sets}), \text{ and } \mathsf{ssSet}_{\leq n}:=\operatorname{Pre}(\Delta_{n}^+,\mathsf{Sets}),\) respectively.
\end{definition}
Directed graphs can be defined in terms of semisimplicial sets.

\begin{example}\label{example: Graph as 1 simplicial set}
A directed graph \(G=(V,E)\) is a \(1\)-truncated semisimplicial set, with \( G_0=V,\ G_1=E, \) and face maps \(d_1(e)=\operatorname{source}(e),  \ d_0(e)=\operatorname{target}(e).\) Conversely, any object of \(\operatorname{Pre}(\Delta_{1}^+,\mathsf{Sets})\) determines a directed graph in this way. If the map \[ (d_1,d_0)\colon G_1\rightarrow G_0\times G_0 \] is injective, one recovers directed graphs with at most one edge from a given source to a given target. In this case, we say that our graph is \emph{simple}. We denote the category of finite simple graphs having exactly one self-loop for each vertex as $\digraphs$.
\end{example}

\begin{excerpt}
Let \(i\colon \Delta^+ \hookrightarrow \Delta \) denote the inclusion. Restriction along \(i\) defines the forgetful functor \( U=i^\ast\colon \mathsf{sSet}\rightarrow \mathsf{ssSet}\), \(X\to X\circ i.\) It admits a left adjoint \(L\colon \mathsf{ssSet} \rightarrow \mathsf{sSet},\) called the \emph{free degeneracy completion}. Explicitly, for \( S \in \mathsf{ssSet}\), \[ L(S)_p = \coprod_{\pi\colon [p]\twoheadrightarrow [q]} S_q,\] where the coproduct is taken over all order-preserving surjections in \(\Delta\) (see also \cite{EbertRandalWilliams2019SemiSimplicial, piccinini}). For \(S\in\mathsf{ssSet}\), we write
\(\widehat S:=UL(S).\) The same construction in degrees at most \(n\) gives an adjunction $L_{\leq n}\dashv U_{\leq n}\colon
\mathsf{ssSet}_{\leq n}
\rightleftarrows \mathsf{sSet}_{\leq n}$. When working in the truncated categories, we suppress the subscript \({\leq n}\) and write \(L\) and \(U\) for \(L_{\leq n}\) and
\(U_{\leq n}\), respectively. Accordingly, for \(S\in\mathsf{ssSet}_{\leq n}\), the notation
\(\widehat S:=UL(S)\) means \(U_{\leq n}L_{\leq n}(S)\).
\end{excerpt}

\begin{excerpt}\label{excerpt: vertex map} 
Let \(S\in \mathsf{ssSet}\). For \(\sigma \in S_p\) and \(0 \le i \le p\), define the \(i\)-th vertex of \(\sigma\) by \( v_i(\sigma) := S(\alpha_i)(\sigma)\in S_0, \) where \( \alpha_i \colon [0] \hookrightarrow[p] \) is the unique order-preserving injection with image \(\{i\}\). The \emph{vertex map} in degree \(p\) is \[ \operatorname{vert}_p\colon S_p \rightarrow (S_0)^{p+1}, \quad \sigma \rightarrow \bigl(v_0(\sigma), \ldots,v_p(\sigma)\bigr).\] We say that \(S\) is \emph{vertex-injective in degree \(p\)} if \(\operatorname{vert}_p\) is injective, and
\emph{vertex-injective} if it is vertex-injective in every degree. In later sections, we will also write \(s(\sigma)\) (\emph{source}) and \(t(\sigma)\) (\emph{tail}) for \(v_0(\sigma)\) and \( v_p(\sigma)\), respectively.
\end{excerpt}

\begin{definition}\label{definition:non-singular-complexes} Let \( \ssreg \) denote the full subcategory of semisimplicial sets \(S\in \operatorname{Pre} (\Delta^+,\mathsf{Sets}) \) such that
\begin{enumerate}[nosep]
    \item $S$ is finite, meaning that $\bigsqcup_pS_p$ is a finite set.
    \item \(S\) is vertex-injective.
    \item Every positive-dimensional simplex has pairwise distinct vertices.
\end{enumerate} Let \(\ssreg_{\leq n}\subset\mathsf{ssSet}_{\leq n}\) be the full
subcategory defined by conditions {\rm (1)--(3)} in degrees \( 0 \leq p \leq n\). We refer to the objects of these categories as \emph{non-singular complexes}. Let \(\ssregh\subseteq\mathsf{ssSet}\) and \(\ssregh_{\leq n}\subseteq\mathsf{ssSet}_{\leq n}\) be the full subcategories spanned by the objects \(\widehat S\), for \(S\in\ssreg\) and \(S\in\ssreg_{\leq n}\), respectively. We call their
objects \emph{extended non-singular complexes}.
\end{definition}

Since the vertex maps of every \(S\in\mathsf{SS}^{\mathrm{ns}}\) are injective, we shall freely regard \(S_p\) as its image under \( \operatorname{vert}_p:S_p\hookrightarrow(S_0)^{p+1}.\) Thus, if \(v=(v_0, \ldots,v_p)\) belongs to this image, we write \([v]=[v_0,\ldots,v_p]\) for the unique simplex having ordered vertex tuple \(v\). Parentheses will be used for ordered vertex words, while brackets indicate the corresponding simplex.

\begin{remark}\label{remark:unique-lift-simplicialmap} Let \(S,T\in\ssreg_{\leq n}\) with \(\widehat S=UL(S)\) and \(\widehat T=UL(T)\). One can see that simplices of \(\widehat T\) are uniquely determined by their repeated vertex words, where repetitions only occur in consecutive blocks. Although \(\widehat g\) is only assumed to be semisimplicial, its underlying degree maps automatically commute with the degeneracy
maps of \(L(S)\) and \(L(T)\). Indeed, let \(\epsilon_i\colon[p+1]\twoheadrightarrow[p]\) be the codegeneracy defining \(s_i\). Since a semisimplicial map commutes with every vertex map, for every
\(\sigma\in L(S)_p\) and \(0 \leq j\leq p+1\), we have \[ v_j\bigl(\widehat g_{p+1}(s_i\sigma)\bigr) =\widehat g_0\bigl(v_j(s_i\sigma)\bigr)= \widehat g_0\bigl(v_{\epsilon_i(j)}(\sigma)\bigr) = v_{\epsilon_i(j)}\bigl(\widehat g_p(\sigma)\bigr) =v_j\bigl(s_i\widehat g_p(\sigma)\bigr).\] Hence \(\widehat g(s_i\sigma)=s_i\widehat g(\sigma)\) so \(\widehat g\) is the underlying semisimplicial map of a unique simplicial map \(g\colon L(S)\rightarrow L(T).\) Together with the adjunction \( L \dashv U\), this gives the natural bijections \[ \operatorname{Hom}_{ \mathsf{ssSet}}( \widehat S, \widehat T) \cong \operatorname{Hom}_{\mathsf{sSet}}( L(S), L(T)) \cong \operatorname{Hom}_{\mathsf{ssSet}}(S,UL(T)).\]\end{remark}

\begin{definition}\label{definition: face poset, hasse diagram associated to a simplicial set}
Let \(S\in\ssreg_{\le n}\). Its \emph{face poset} \(P_{S} \) is \(\bigsqcup_{p=0}^n S_p\) with \( \sigma \leq \tau\) if and only if \( \sigma \) is obtained from \( \tau \) by applying a sequence of face maps. The \emph{Hasse diagram} \(\mathcal H(S)\) is the directed graph whose vertices are the elements of \(P_{S}\), with an arrow \( \sigma \rightarrow \tau \) whenever \( \sigma \) is a codimension-one face of \(\tau\).
\end{definition}

\begin{definition}\label{definition: cellular (co)sheaves on a simplicial set}
Let \(S\in\ssreg_{\leq n}\). A \emph{cellular sheaf}
on \(S\) is a covariant functor \( F\colon P_{S} \rightarrow \mathrm{Vect}_\Bbbk. \) A morphism of cellular sheaves is a natural transformation. We denote the resulting category by \( \mathsf{Sh}(S). \)
\end{definition}

\begin{excerpt}\label{excerpt:assignments for alexandrov and preorders}
The data of a cellular sheaf on \(S\) is equivalent to the datum of a sheaf on the topological space \(X_{P_S}\) obtained by considering $P_{S}$ endowed with the Alexandrov topology of upper order ideals, that is the topology generated by the base $\cB_S$ consisting of the following subsets of $ P_{S}$ (see \cite{Ayzenberg2025SheafTheory, curry2014SheavesCosheavesAndApplications, fioresi2026sheaves, Sancho2018HomotopyFiniteRingedSpaces}) :
\[ U_\sigma:=\{\tau\in P_{S}\mid \sigma\leq\tau\},\qquad \sigma\in P_{S}.\]
Since every open cover of \(U_\sigma\) by elements of \(\cB_S\) contains
\(U_\sigma\) itself, every presheaf on the basis $\cB_S$ automatically satisfies the sheaf condition on this basis, see \cite{vakil} for this notion. As a consequence, we have the following standard result (see \cite{Ayzenberg2025SheafTheory, fioresi2026sheaves} for more general discussions).

\begin{proposition}\label{proposition: cellular sheaves Alexandrov equivalence}
The assignments from \ref{excerpt:assignments for alexandrov and preorders} define a natural equivalence \(\mathsf{Sh}(S) \simeq \mathsf{Sh}(X_{P_S},\mathrm{Vect}_\Bbbk),\) where $\mathsf{Sh}(X_{P_S},\mathrm{Vect}_\Bbbk)$ denotes the category of sheaves of vector spaces on $X_{P_S}$.
\end{proposition}
The Alexandrov perspective bridges the diagrammatic definition of cellular sheaves to sheaves on a topological space, and thereby to finite ringed spaces, quasi-coherent modules, and gluing data in the sense of \cite{fioresi2026gluing, fioresi2026sheaves,Sancho2017FiniteSpacesSchemes, Sancho2018HomotopyFiniteRingedSpaces}. These comparisons are discussed in \S \ref{section: connection sheaves and connections}.
\end{excerpt}

\begin{remark}
The face poset is attached to \(S\), not to \(\widehat S\). Indeed, the additional simplices introduced by the free degeneracy completion do not represent new geometric cells. Including them in the face poset would produce additional open subsets in the associated Alexandrov space without adding new geometric information.
\end{remark}

\subsection{Differential calculi and differential graded algebras}
We recall first-order differential calculi and their extension to exterior
differential graded algebras. Throughout this subsection, \(A\) denotes a unital \(\Bbbk\)-algebra, not necessarily commutative.

\begin{definition} A \emph{first-order differential calculus} (FODC) \((\Omega^1_A,d)\) on \(A\) consists of an \(A\)\nobreakdash–\(A\)-bimodule \(\Omega_A^1\) together with a \(\Bbbk\)-linear \emph{exterior derivative}
\( d\colon A \rightarrow \Omega^1_A\), satisfying the \emph{Leibniz rule}
\[d(fg)=(df)g + f(dg),
\quad \forall f,g\in A,\] and the \emph{surjectivity condition}, which states that \(\Omega_A^1\) is spanned by elements of the form \( f dg\), for \(f,g\in A\). A \emph{morphism of first-order differential calculi}
\((A, \Omega^1_A,d_A)
\rightarrow
(B, \Omega^1_B,d_B)\)
is given by an algebra morphism \(\varphi\colon A\rightarrow B\) together with an \(A\)\nobreakdash–\(A\)-bimodule map \( \varphi^* : \Omega^1_A \rightarrow\Omega^1_B\)
such that \(\varphi^* \circ d_A= d_B \circ \varphi.\) Here \(\Omega^1_B\) is seen as an \(A\)\nobreakdash–\(A\)-bimodule via \(\varphi\). Additionally, a FODC \(\Omega^1_A\) is called \emph{inner} if there exists an element \(\theta\in\Omega^1_A\) such that
\(d a  =  [\theta,a] = \theta a - a\theta, 
\text{ for any }a \in A\). It is called \emph{connected} if
\(\ker(d) = \Bbbk\cdot 1_A\).

\end{definition}
We usually omit the differential from the notation and write \(\Omega^1_A\) or \(\Omega^1\) if the underlying algebra is clear. The category with FODCs over algebras with arrows given by morphisms as defined above is denoted by \( \mathsf{DC}\). The following is well-known, see for instance \cite[Prop 1.5]{BeggsMajid2020QRG}.
\begin{lemma}\label{lemma: universal calculus}
Let \(A\) be a \(\Bbbk\)-algebra. Then:
\begin{enumerate}[nosep]
\item \(A\) admits a universal connected FODC \((\Omega^1_{\mathrm{uni}}(A),d_{\mathrm{uni}})\) given by
\[
\Omega^1_{\mathrm{uni}}(A)  \coloneq  \ker(\mu)\subseteq A\otimes_\Bbbk A,
\qquad d_{\mathrm{uni}}(a)=1_A\otimes a -a\otimes 1_A, \quad \forall a\in A, \]
where \(\mu\colon A\otimes_\Bbbk A\rightarrow A\) denotes the algebra multiplication.
\item Any FODC \(\Omega^1_A\) is isomorphic to a quotient \(\Omega^1_{\mathrm{uni}}(A)/N\), for some sub-bimodule \(N\subseteq\Omega^1_{\mathrm{uni}}(A)\), with differential induced by \(d_{\mathrm{uni}}\). 
\end{enumerate}
\end{lemma}

\begin{example}\label{example: graph differential calculus}
A fundamental example for this text is the equivalence between finite directed graphs and (FODCs) on algebras of functions over finite sets. We outline the correspondence and come back to this in more detail in \(\S\ref{section: from semisimplicial sets to differential graded algebras}\). Let \(V\) be a finite set and \(A_V = \Bbbk(V)\) the algebra of \(\Bbbk\)\nobreakdash-valued functions on \(V\) with pointwise multiplication, i.e. \(A_V \coloneq \text{span}_\Bbbk\{\delta_x \mid x\in V\},\) where \(
\delta_x(y) = \begin{cases}
1,&y=x,\\
0,&y\neq x
\end{cases}\) are orthogonal idempotents. Let \(G=(V,E)\) be a simple directed graph admitting exactly one self-loop per vertex. Write \(\overline{E}\) for the collection of non-self loops, i.e. \( \overline{E} = \{(x,y)\in E, \ x\neq y\}\). Then \(\Omega_G^1  : = \text{span}_\Bbbk\{\omega_{x\rightarrow y}\mid (x,y)\in \overline{E}\}\) is an \(A_V\)\nobreakdash–\(A_V\)-bimodule by \[ f\cdot\omega_{x\rightarrow y} = f(x) \omega_{x\rightarrow y},
\quad \omega_{x\rightarrow y}\cdot f = \omega_{x\rightarrow y}f(y), \quad f\in A_V.\]
The differential \( d\colon A_V\rightarrow \Omega_G^1\) is defined on a basis element \(\delta_x\) by
\[ d \delta_x = \sum_{\substack{y\rightarrow x \in \overline{E}}}\omega_{y \rightarrow x}
-\sum_{\substack{x\rightarrow y \in \overline{E} }}\omega_{x\rightarrow y}, \]
or, for \(f = \sum_{x} f(x) \delta_x\in A_V\),  \( df =\sum_{x\rightarrow y\in \overline{E}}\bigl(f(y)-f(x)\bigr)\omega_{x\rightarrow y}.\) One verifies that \( d \) satisfies the Leibniz rule, and, since \(\delta_x\cdot d(\delta_y)=\omega_{x\rightarrow y}\) for any \((x\rightarrow y)\in \overline{E}\), \( \Omega^1_G \) also satisfies the surjectivity condition. Hence, \(\Omega_G^1\) is a FODC on \(A_V\). Moreover, this differential calculus is inner, since \(\theta \coloneq \sum_{(x\rightarrow y)\in \overline{E}}\omega_{x\rightarrow y}\) satisfies
\(d a=[\theta,a]\) for all \(a\in A_V\). 

Conversely, any FODC \((\Omega^1,d) \) over \(A_V\) determines a directed graph. Indeed, the vertex set is \(V\) and can be recovered from the algebra \(A_V\), while \[ \overline{E}_\Omega \coloneq \{(x,y)\in V\times V\mid x\neq y,\  \Omega^1_{x,y}\neq 0\},\] where \(\Omega^1_{x,y} \coloneq \delta_x\cdot \Omega^1\cdot \delta_y\). The surjectivity condition implies that \( \dim_{\Bbbk}\bigl(\Omega^1_{x,y}\bigr)\leq 1\) for \(x\neq y\) (an explicit proof is given in \ref{lemma: fibers of FODC are 1 dimensional} below), while \(\Omega^1_{x,x}=0\). Thus \(\overline{E}_\Omega\) defines a simple directed graph without loops. Adjoining the self-loops \((x,x)\) at every \( x \in V\) produces the corresponding object of \( \digraphs \). Moreover, the differential is \emph{canonical} in the following sense. For every $ (x,y) \in \overline{E}_\Omega $, the element \( \omega_{x\rightarrow y}\  \coloneq \ \delta_x d(\delta_y) \in \Omega^1_{x,y}\) spans \(\Omega^1_{x,y}\), and the idempotent identities combined with $d(1_{A_V})=0$ force the incidence differential
\begin{equation}\label{equation: where the differential of a FODC lives to in terms of idempotent decomps}
d(\delta_t) = \sum_{(y,t)\in \overline{E}}\omega_{y\rightarrow t}-\sum_{(t,y)\in \overline{E}}\omega_{t\rightarrow y}\in \big(\bigoplus_{u}\Omega^1_{u, t} \big) \oplus \big(\bigoplus_{w} \Omega^1_{t,w}\big),
\quad \text{for all }t\in V.
\end{equation}
These assignments give an anti-equivalence of categories. Recall that we denote by \(\digraphs\) the category whose objects are finite directed graphs \((V,E)\) with at most one arrow for each ordered pair of vertices, and where each vertex comes equipped with a self‑loop. Morphisms \(\psi\colon(W,F)\to(V,E)\) in \(\digraphs\) are maps \(\psi\colon W\to V\) such that
\[ (\psi(w),\psi(z))\in E \qquad  \text{for every }(w,z)\in F. \] Since every vertex carries a self-loop, this is equivalent to requiring that, for every \((w,z)\in F\), either \(\psi(w)=\psi(z)\) or \((\psi(w),\psi(z))\in\overline E\). Self-loops are present in the graph category but are not realized as 1-forms. Their role is to express edge-collapsing morphisms as maps that send arrows to arrows. As shown in \cite[\S5]{Majid2012Graph}, the above assignment \( G \rightarrow \Omega^1_G\) extends to a fully faithful functor \(F: (\digraphs)^\mathrm{op}\rightarrow \mathsf{DC}\). Its essential image is the full subcategory of FODCs over algebras of functions on finite sets. Consequently, it induces an anti-equivalence between \(\digraphs\) and this subcategory.
\end{example}

\begin{definition}\label{definition: DGA}
Let \(A\) be a \(\Bbbk\)‑algebra. A \emph{differential graded algebra} (DGA) \((\Omega^\bullet,d,\wedge)\) on \(A=\Omega^0\) consists of a graded \(\Bbbk\)-vector space \( \Omega^\bullet =\bigoplus_{n\ge0}\Omega^n,\) with a unital associative, graded product \( \wedge \colon \Omega^p \otimes_\Bbbk \Omega^q \rightarrow \Omega^{p+q},\)
together with a \(\Bbbk\)‑linear differential \(d\colon \Omega^\bullet \rightarrow \Omega^{\bullet+1}\) satisfying \(d^2 = 0\) and the \emph{graded Leibniz rule} \[ d(\alpha \wedge\beta)
= (d\alpha)\wedge\beta + (-1)^{\lvert\alpha\lvert}\alpha\wedge(d\beta) ,
\quad \text{for all homogeneous } \alpha,\beta \in \Omega^\bullet . \] A \emph{morphism} of DGAs 
\(\varphi\colon(\Omega_A^\bullet,d_A,\wedge_A)\rightarrow(\Omega_B^\bullet,d_B,\wedge_B)\)
is a unital degree‑preserving algebra morphism that commutes with the differential. 
The category of all DGAs and their morphisms is denoted \(\mathsf{DG}\). For \(n\geq 0\), \( \mathsf{DG}_n \subseteq\mathsf{DG}\) denotes the full subcategory of DGAs of degrees at most \(n\).
\end{definition}
As \((\Omega^\bullet, \wedge)\) is an algebra, all \(\Omega^p\) are naturally \(A\)\nobreakdash–\(A\)-bimodules by \(a \cdot \alpha \cdot b= a\wedge \alpha \wedge b \in \Omega^p\). Associativity also readily gives that the wedge product is \(A\)-balanced. 

\begin{definition}\label{definition: graded commutative and exteriority of a DGA}
A differential graded algebra \(\Omega^\bullet\) is \emph{graded commutative} if \(\alpha\wedge \beta =\ (-1)^{\lvert \alpha\lvert \lvert \beta\lvert }\beta\wedge \alpha, \text{ for all homogeneous } \alpha, \beta \in \Omega^\bullet\). It is \emph{exterior} if \(\Omega^1\) satisfies the surjectivity condition, and the iterated wedge 
\[\wedge\colon (\Omega^1)^{\otimes_An}
\rightarrow \Omega^n \] is surjective for each \(n\ge1\). 
For \((\Omega^\bullet, \wedge,d)\) a DGA, a \emph{differential graded ideal} \(I^\bullet \subseteq\Omega^\bullet\) is a graded two-sided ideal such that \[
\Omega^p \wedge I^q \subseteq I^{p+q},\quad
I^p \wedge \Omega^q \subseteq I^{p+q},\quad
d(I^q)\subseteq I^{q+1}\quad\text{for all }p,q\ge 0.
\]
A \emph{differential graded left $\Omega^\bullet$-module} (DG-module) is a graded left $\Omega^\bullet$-module \((M^\bullet, \cdot, d_M )\), where \(d_M: M^\bullet \rightarrow M^{\bullet+1}\) is a degree one \(\Bbbk\)-linear map satisfying \(d_M^2=0\), as well as \(d_M(\alpha\cdot m)= d(\alpha)\cdot m + (-1)^{\lvert \alpha\lvert } \alpha\cdot d_M(m), \ \text{for all homogeneous } \alpha\in \Omega^\bullet,\ m \in M^\bullet.\) \end{definition}
Right and bimodule versions are defined analogously. A morphism of DG-modules is a degree-zero $\Omega^\bullet$-linear map $f\colon M^\bullet\rightarrow N^\bullet$ commuting with the differential. For \(I\) a differential ideal in \(\Omega^\bullet\), the quotient \(\Omega^\bullet/I\) is again a DGA.

\begin{example}\label{example: universal exterior DGA}
Let \(A\) be a \(\Bbbk\)‑algebra. Its \emph{universal exterior differential graded algebra} \((\Omega^\bullet_{\text{uni}}, \wedge, d_{\mathrm{uni}})\) is defined as 
\[ \Omega^\bullet_{\text{uni}}(A) = \bigoplus_{n\ge0}\Omega^n_{\text{uni}}(A), \qquad \Omega^0=A, \qquad \Omega^n_{\text{uni}}(A) = \bigcap_{i=0}^{n-1} \ker\bigl(\mu_i\colon A^{\otimes(n+1)}\rightarrow A^{\otimes n}\bigr),\] where \(\mu_i\) multiplies the \(i\)th and \((i+1)\)th factors in the tensor product. On elementary tensors,
\[ (a_0\otimes\ldots\otimes a_p)\wedge(b_0\otimes\ldots\otimes b_q) = a_0\otimes\ldots\otimes (a_p b_0)\otimes\ldots\otimes b_q, \]
\[d_{\text{uni}}\bigl(a_0\otimes\ldots\otimes a_p\bigr)
=\sum_{i=0}^{p+1}(-1)^i a_0\otimes\ldots\otimes a_{i-1}\otimes 1_A\otimes a_i\otimes\ldots\otimes a_p, \]
and \(d_{\text{uni}}\) is extended by the graded Leibniz rule. This is the largest exterior DGA on \(A\), in the sense that every other exterior DGA on \(A\) is a quotient of \(\Omega^\bullet_{\text{uni}}(A)\) by a differential graded ideal. See \cite[Chapter 1]{Loday1998CyclicHomology} for details. We also recall the following lemma, and use the two descriptions of \(\Omega^\bullet_{\text{uni}}\) interchangeably in later paragraphs (see for instance \cite{BeggsMajid2020QRG} for more details).
\end{example}

\begin{lemma}\label{lemma: isomorphism of universal differential graded algebra}
For a finite-dimensional \(\Bbbk\)-algebra \(A\), \(n \geq 1,\) \(\Omega^{n}_{\mathrm{uni}}(A)\cong (\Omega^1_{\mathrm{uni}}(A))^{\otimes_An}\).
\end{lemma}

Any first‑order calculus \(\Omega^1\cong\Omega^1_{\rm uni}(A)/N\) on \(A\) has a  \emph{maximal prolongation} to an exterior DGA \(\Omega^\bullet_{\max}\)  with \(\Omega^1_{\max}\cong\Omega^1\). This prolongation is universal in  the sense that any exterior DGA \((\widetilde{\Omega}^\bullet,d')\) extending \((\Omega^1,d)\) admits a unique surjective DGA map  \( \Omega^\bullet_{\max} \twoheadrightarrow \widetilde{\Omega}^\bullet \). Concretely, \[ \Omega^\bullet_{\max} \cong \Omega^\bullet_{\rm uni}(A)/I, \] where \(I\) is the differential ideal in \(\Omega^\bullet_{\rm uni}(A)\) generated by \(N\subset\Omega^1_{\rm uni}(A)\) 
and \(d(N)\subset\Omega^2_{\rm uni}(A)\).

\begin{example}\label{example: maximal prolongation of a graph calculus}
Let \(V\) be a finite set and let \(A=\Bbbk(V)\). Under the anti-equivalence of Example \ref{example: graph differential calculus}, the universal FODC \(\Omega^1_{\mathrm{uni}}(A)\) corresponds to the complete directed graph on \(V\), with one arrow between every
ordered pair of distinct vertices and with the formal self-loops required in \(\digraphs\). By Lemma \ref{lemma: isomorphism of universal differential graded algebra}, the space \(\Omega^n_{\mathrm{uni}}(A) \cong \bigl(\Omega^1_{\mathrm{uni}}(A)\bigr)^{\otimes_A n} \) is spanned by composable directed paths of length \(n\) in this graph.

More generally, let \(G=(V,E)\) be a finite directed graph with associated graph calculus \( \Omega^1_G \cong \Omega^1_{\mathrm{uni}}(A)/N,\) where \(N\) is spanned by the basis forms corresponding to the arrows not contained in \(G\). Quotienting only by the graded ideal generated by \(N\) gives \[ T_A(\Omega^1_G) \cong \Omega^\bullet_{\mathrm{uni}}(A)/\langle N\rangle \cong A\oplus\Omega^1_G \oplus(\Omega^1_G)^{\otimes_A 2}  \oplus\ldots .\] Its degree-\(n\) basis elements are the directed paths of length \(n\) in \(G\). To obtain a differential graded algebra, one must additionally impose the relations generated by \(d_{\mathrm{uni}}(N)\). Thus, if \(I = \bigl\langle N,d_{\mathrm{uni}}(N) \bigr \rangle \subseteq \Omega^\bullet_{\mathrm{uni}}(A),\) the maximal prolongation of \(\Omega^1_G\) is \( \Omega^\bullet_{G,\max} = \Omega^\bullet_{\mathrm{uni}}(A)/I. \) In particular, for every \(n\geq1\), \( \Omega^n_{G,\max} \cong \Omega^n_{\mathrm{uni}}(A)/ \bigl(I \cap \Omega^n_{\mathrm{uni}}(A)\bigr). \) 
\end{example}

\section{Non-singular complexes and differential graded algebras}\label{section: from semisimplicial sets to differential graded algebras}
Let \(S\in\ssreg_{\leq n}\) with \(\widehat S:=UL(S)\). In this section, we associate to \(\widehat S\) the normalized cochain DGA \( \Omega^\bullet (\widehat S) := N^\bullet(L(S);\Bbbk).\) Since normalization removes the degenerate simplices adjoined by \(L\), \(\Omega_S^p\) has a basis indexed by the simplices \(S_p\). Hence, we will denote this DGA by \(\Omega_S^\bullet.\) 
These assignments define a contravariant functor
\(\Omega^\bullet\colon \bigl(\ssregh_{\leq n}\bigr)^{\mathrm{op}} \rightarrow \mathsf{DG}_n.\)
We proceed to prove that this functor is fully faithful and we characterize its essential image proving the first main result of our paper. We start with a lemma.

\begin{lemma}\label{lemma: simplices inject into clique complex}
Let \(S\in\ssreg_{\leq n}\). For every
\(1\leq p\leq n\), the image of the vertex map lies in the \((p+1)\)-clique complex
of the underlying directed graph, i.e.
\[ \operatorname{vert}_p(S_p) \subseteq \left\{ (v_0, \ldots, v_p) \in (S_0)^{p+1} \ \middle|\ (v_i,v_j)\in S_1\text{ for all }0\leq i<j\leq p \right \}. \]
\end{lemma}

\begin{proof}
Let \(\sigma\in S_p\), with \(\operatorname{vert}_p(\sigma)=(v_0,\ldots,v_p)\), and fix
\(0\leq i<j\leq p\). Let \(\beta_{i,j}\colon[1]\hookrightarrow[p]\) be the order-preserving
injection with \(\beta_{i,j}(0)=i\) and \(\beta_{i,j}(1)=j\). Then \(S(\beta_{i,j})(\sigma)\in S_1\) has ordered vertex tuple \((v_i,v_j)\). Hence \([v_i,v_j]\in S_1\).
\end{proof}

\begin{excerpt}\label{excerpt: Object part of functor D}
Let \(S\in\ssreg_{\leq n}\). Generalizing the constructions of Example \ref{example: graph differential calculus}, we define \( \Omega^\bullet_{S}\) by \[ \Omega^0_{S} \coloneq A_{S} := \Bbbk(S_0), \qquad  \Omega_{S}^p \coloneq \bigoplus_{ \sigma \in S_p} \Bbbk \rho_\sigma, \quad 1 \leq p \leq n, \qquad \Omega^q_{S} = 0 , \quad q > n.\] For \(v\in S_0\), we set \(\rho_{[v]}:=\delta_v\), and for an edge \(e\in S_1\), we also write \(\omega_e:=\rho_e\). If \(\sigma=[v_0,\ldots,v_p]\), then the \(A_{S}\)-bimodule structure is given by \[ f\cdot \rho_\sigma\cdot g = f(v_0)\rho_\sigma g(v_p), \quad f,g\in A_{S}. \] For simplices \(\sigma= [v_0 ,\ldots ,v_p]\in S_p \), \( \tau = [w_0 ,\ldots ,w_q]\in S_q \), set \[ \rho_\sigma\wedge\rho_\tau \coloneq 
\begin{cases}
\rho_{\sigma \star \tau}, 
& \text{if }v_p=w_0\text{ and } \sigma \star \tau \coloneq [v_0, \ldots,v_{p-1}, v_p=w_0,w_1,\ldots ,w_q]\in S_{p+q} ,\\
0, & \text{otherwise},
\end{cases} \]
and extend to a \(\Bbbk\)-linear map \(\Omega_{S}^p\otimes \Omega_{S}^q\rightarrow\Omega_{S}^{p+q}\). The differential \(d\colon  \Omega_{S}^p  \rightarrow \Omega_{S}^{p+1}\) is defined on basis elements by
\[ d(\rho_\sigma) = \sum_{j=0}^{p+1}(-1)^j \sum_{\substack{\tau\in S_{p+1} \\d_j\tau=\sigma}} \rho_\tau, \]
and extended \(\Bbbk\)‑linearly. In degree \(n\), the differential is zero.
\end{excerpt}
 
For $X:\Delta^{op}\rightarrow \Sets$ a simplicial set, the cochains \( C^p(X;\Bbbk) \coloneq \mathrm{Hom}_{\Bbbk}\bigl(\Bbbk[X_p],\Bbbk\bigr) \) form a cosimplicial vector space, and the associated normalized cochain complex carries the cup product, and hence defines a differential graded algebra, denoted by \( N^\bullet(X;\Bbbk).\) For \(S \in \ssreg_{\leq n},\) the nondegenerate simplices of \(L(S)\) are precisely the simplices of \(S\). Hence, the construction above gives an identification \( \Omega^\bullet_{S}\cong N^\bullet(L(S);\Bbbk).\) Under this identification, the normalized cochain differential and
the cup product agree respectively with the differential and product defined above on \(\Omega_S^\bullet\). As a consequence, we have the following.

\begin{proposition}\label{proposition: normalized cochains concrete dga}
The DGA \(\Omega^\bullet_{S}\) is isomorphic to \( N^\bullet(L(S);\Bbbk) \) with the cup product. In addition, \(\widehat{S} \to \Omega^\bullet_{S}\) defines a contravariant functor \( \Omega^\bullet\colon \bigl( \ssregh_{\leq n}\bigr)^{\mathrm{op}} \to \mathsf{DG}_{\leq n}.\)
\end{proposition}

\begin{proof} We only describe the functor $\Omega^\bullet$. Let \( \widehat{g}\colon \widehat{S}\rightarrow \widehat{T} \) be a morphism in \(\ssregh_{\leq n}\). Equivalently, by Remark \ref{remark:unique-lift-simplicialmap}, \( \widehat{g} \) is the underlying semisimplicial map of a simplicial map \( g \colon L(S) \rightarrow L(T).\) We define \( \Omega^\bullet(\widehat g) \coloneq g^\ast\colon N^\bullet(L(T);\Bbbk) \rightarrow N^\bullet(L(S); \Bbbk). \) Under the above identifications, this gives a DGA morphism \( \Omega^\bullet(\widehat g) = g^\ast \colon \Omega_T^\bullet \rightarrow \Omega_S^\bullet.\) This assignment is functorial and contravariant. 
\end{proof}

\begin{excerpt}
For a morphism \( \widehat g\colon \widehat{S}\rightarrow \widehat{T} \), we explicitly describe \(\Omega^\bullet(\widehat{g})=g^*\). In degree zero, \( g^\ast( \delta_t) = \sum_{s\in S_0, \ g_0(s) = t}\delta_s, \ t\in T_0.\) For \( e=[a,b]\in T_1\), \[ g^\ast(\omega_{[a,b]}) = \sum_{\substack{[x,y]\in S_1\\ g_0(x)=a,\ g_0(y)=b}} \omega_{[x,y]},\] and for \(\tau = [w_0,\ldots,w_p]\in T_p, \) \[ g^\ast(\rho_\tau) = \sum_{\substack{\sigma=[v_0,\ldots,v_p]\in S_p\\
g_0(v_i)=w_i,\ \forall i}} \rho_\sigma.\] Since \(g^\ast\bigl(\sum_{t\in T_0}\delta_t\bigr) = \sum_{s\in S_0}\delta_s\), this morphism is unital.
\end{excerpt}

Before establishing fully faithfulness, we relate \(\Omega_S^\bullet\) to two other algebraic objects, namely the universal differential algebra of \(A_S\) and the incidence algebra of the face poset \(P_S\).

\begin{remark}\label{remark: relation-with-max-prol}
Let \(V\) be a finite set with \( A_V :=\Bbbk(V) \). Define the complete non-singular \( n \)-complex \( S^{\max}_{V ,\leq n}\) by \[ \bigl( S^{\max}_{V, \leq n}\bigr)_p := \left\{ (v_0, \ldots,v_p)\in V^{p+1} \ \middle|\ v_i \neq v_j \text{ for } i\neq j \right\}, \qquad 0\leq p\leq n, \] with face maps given by deletion of vertices. If \( S \in \mathsf{SS}^{\mathrm{ns}}_{\leq n}\) has \( S_0 = V\), its vertex maps define the embedding \( \iota_S: S \hookrightarrow S^{\max}_{V , \leq n} \). After passing to free degeneracy  completions, \( \Omega^\bullet \) gives a surjective DGA morphism \( \iota_S^*: \Omega^\bullet_{S^{\max}_{V,\leq n}} \twoheadrightarrow \Omega^\bullet_S\) with kernel \[ J_S = \bigoplus_{p=1}^n \operatorname{span}_{\Bbbk} \left
\{ \rho_{ [v] } \ \middle| \ v\in \bigl(S^{\max}_{V,\leq n}\bigr)_p \setminus\operatorname{vert}_p(S_p) \right
\}, \] and hence \(\Omega^\bullet_S \cong \Omega^\bullet_{S^{\max}_{V,\leq n}}/J_S.\) Thus \(S\) is obtained from the complete non-singular complex by deleting simplices, while \(\Omega^\bullet_S\) is obtained by quotienting out the corresponding basis forms.

To compare this with the universal differential graded algebra, consider \[ P_p(V) = \left\{ (v_0, \ldots, v_p) \in V^{p+1} \ \middle|\ v_{i-1} \neq v_i \right\},\] and put \( m(v) := \omega_{[v_0,v_1]}\wedge\ldots\wedge
\omega_{[v_{p-1},v_p]}. \) The elements \(m(v)\), with \(v\in P_p(V)\), form a basis of \(\Omega^p_{\mathrm{uni}}(A_V)\). The canonical surjection \( \pi_S: \Omega^\bullet_{\mathrm{uni}}(A_V) \twoheadrightarrow \Omega^\bullet_S \) is given by \[ \pi_S(m(v)) = \begin{cases} \rho_\sigma, & \operatorname{vert}_p(\sigma)=v
\text{ for some }\sigma\in S_p,\\ 
0, & \text{otherwise}.\end{cases}\]
Consequently, \( \Omega^\bullet_S \cong \Omega^\bullet_{\mathrm{uni}}(A_V)/I_S,\) where \[ I_S = \bigoplus_{p\geq1} \operatorname{span}_{\Bbbk} \left\{ m(v) \ \middle| \ v\in P_p(V),\ v\notin\operatorname{vert}_p(S_p) \right\}, \] where \(S_p=\varnothing\) for \(p>n\).

Finally, if \( N = \ker\bigl( \Omega^1_{\mathrm{uni}}(A_V)\twoheadrightarrow\Omega^1_S
\bigr), \) then \( \langle N,d_{\mathrm{uni}} N \rangle \subseteq I_S\), and the preceding
surjection factors through the maximal prolongation of the calculus on the \(1\)-skeleton:
\[ \Omega^\bullet_{\max}(\Omega^1_S) = \frac{\Omega^\bullet_{\mathrm{uni}}(A_V)}
{\langle N,d_{\mathrm{uni}}N\rangle}\twoheadrightarrow \Omega^\bullet_S. \] Accordingly, deleting edges determines the first-order calculus, whereas the remaining part of the quotient records the genuinely higher-dimensional choice of which directed paths are filled by simplices.
\end{remark}

\begin{remark}\label{remark: relation between AW DGA and the incidence algebra} There is a close relationship between \(\Omega^\bullet_{S}\) and the \emph{incidence algebra \(\Bbbk[P_{S}]\)} associated to the face poset \(P_{S}\) of \(S\in \ssreg_{\leq n}\) (see e.g. \cite[\S 3.6]{Stanley2011EC1} or \cite{Yanagawa2005Dualizing} for more on incidence algebras). The algebra \(\Bbbk[P_{S}]\) is defined as the \(\Bbbk\)-algebra with basis $\{e_{\xi,\tau}\mid \xi\leq\tau\}$, product  \(e_{\xi,\tau}e_{\nu,\zeta}=\delta_{\tau, \nu}e_{\xi,\zeta}\), and unit \(\sum_{\xi\in P_{S}}e_{\xi,\xi}\). 
For each simplex $\sigma$, put \[ T_\sigma \coloneq \sum_{\substack{\tau\  \text{with } \\\sigma \star\tau\ \text{defined}}} e_{\tau, \sigma\star\tau} \in \Bbbk[P_{S}]. \]
The family $\{T_\sigma\}_{\sigma\in P_{S}}$ is linearly independent as the \(T_\sigma\) have disjoint support in the incidence basis. Moreover, \[T_\sigma T_\tau=
\begin{cases}
T_{\tau\star\sigma},&\text{if }\tau\star\sigma\ \text{is defined},\\[2pt]
0,&\text{otherwise}.
\end{cases}\]
For $C \coloneq \mathrm{span}_\Bbbk\{T_\sigma\}\subset \Bbbk[P_{S}]$ as subalgebra of \(\Bbbk[P_{S}]\),
the map
\[ \Psi:\ (\Omega_S^\bullet,\wedge) \rightarrow \Bbbk[P_{S}], \quad \Psi(\rho_\sigma) \coloneq T_\sigma, \] is an injective anti-algebra morphism with image $C$. Equivalently, after forgetting the grading, it induces an algebra isomorphism
\((\Omega_S^\bullet)^{\mathrm{op}} \cong C\subseteq\Bbbk[P_S].\)
\end{remark}

Let \(\mathsf{Im}_n\subseteq\mathsf{DG}_n\) denote the full subcategory spanned by the objects
\(\Omega^\bullet(\widehat S)=\Omega_S^\bullet, \ \widehat S=UL(S)\in\ssregh_{\leq n}, \ S\in\ssreg_{\leq n}.\) 
For a map \(\widehat g\colon\widehat S\rightarrow\widehat T,\) let \(g\colon L(S)\rightarrow L(T) \) denote its unique simplicial lift as in Remark \ref{remark:unique-lift-simplicialmap}. We now prove that the assignment \(\widehat g\to\Omega^\bullet(\widehat g)=g^\ast\) induces natural bijections
\[ \mathrm{Hom}_{\ssregh_{\leq n}} (\widehat S, \widehat T) \cong \mathrm{Hom}_{\mathsf{Im}_n} \bigl(\Omega_T^\bullet,\Omega_S^\bullet\bigr).\] 
To that end, we make use the following extension criterion.
\begin{lemma}\label{lemma: vertex maps into completed cores}
Let \(S,T\in\ssreg_{\leq n}\), with
\(\widehat S=UL(S)\) and \(\widehat T=UL(T)\).
A map \(g_0\colon S_0\to T_0\) extends uniquely to a morphism \(\widehat g\colon\widehat S\rightarrow\widehat T\) if and only if, for every \(\sigma=[v_0,\ldots,v_p]\in S_p\), the compression of \([g_0(v_0),\ldots,g_0(v_p)]\) is a simplex of \(T\). In that case,
\(\widehat g([v_0,\ldots,v_p]) =[g_0(v_0),\ldots,g_0(v_p)]. \)
\end{lemma}
\begin{proof}
A simplex of \(\widehat T=UL(T)\) is uniquely determined by its vertex word, and such a word represents a simplex of \(\widehat T\) when its compression is a simplex of \(T\). Hence the above formula defines \(\widehat g\) exactly under the stated condition. It commutes with the face maps because these delete entries of vertex words, and uniqueness follows because simplices of \(\widehat T\) are determined by their vertex words.
\end{proof}

\begin{proposition}\label{proposition: omega functor faithful}
The functor \( \Omega^\bullet \colon \bigl(\ssregh_{\leq n}\bigr)^{\mathrm{op}} \to \mathsf{Im}_n\)
is faithful.
\end{proposition}

\begin{proof} Let \(\widehat g,\widehat h:\widehat S\to \widehat T\) be morphisms such that
\(g^\ast=h^\ast\). It suffices to compare degree \(0\). For \(t\in T_0\), \[ g^\ast(\delta_t) = \sum_{\substack{s\in S_0\\ g_0(s)=t}}\delta_s, \qquad h^\ast(\delta_t) = \sum_{\substack{s\in S_0\\ h_0(s)=t}}\delta_s.\]
Since the \(\delta_s\)'s form a basis, \(g_0^{-1}(t)=h_0^{-1}(t)\) for every
\(t\), hence \(g_0=h_0\). By the preceding lemma, a morphism \(\widehat S\to \widehat T\) is uniquely determined by its degree-zero map. Thus \(\widehat g=\widehat h\).
\end{proof}

\begin{proposition}\label{proposition: omega functor full}
The functor \( \Omega^\bullet\colon
\bigl(\ssregh_{\leq n}\bigr)^{\mathrm{op}}
\to \mathsf{Im}_n \) is full.
\end{proposition}

\begin{proof}
Up until degree \(1\), our argument is similar to the one given in \cite{BeggsMajid2020QRG}. We include the details because they form the basis of the higher-dimensional argument.

Let \( \phi \colon \Omega^\bullet_{T }\to \Omega^\bullet_{S}\) be a DGA morphism. We construct a unique morphism \(\widehat g\colon \widehat{S}\to \widehat{T}\) such that \(\phi=g^\ast.\)

First, since \( \phi^0: \Omega_T^0\rightarrow \Omega_S^0\) defines a unital algebra morphism, the collection \(\{ \phi^0(\delta_t) \}_{t\in T_0}\) is a mutually-orthogonal set of idempotents in \(\Omega_S^0\) summing to \(1_S\in \Omega_S^0\). Hence, \(S_0 = \bigsqcup_{t\in T_0 } A_t,\) where \(A_t = \{s\in S_0  \ \lvert \ \phi^0(\delta_t)(s)=1_{\Bbbk}\}\), and for each \(s\in S_0 \) there exists a unique \( t_s \in T_0 \) such that \(s \in A_{t_s}\). This gives a well-defined map \( g_0: S_0 \rightarrow T_0,\) \(s \rightarrow t_s\), and, by construction \( \phi^0( \delta_t) = \sum_{s\in A_t}\delta_s = \sum_{\{s \in S_0, \ g_0(s)=t \}}\delta_s. \)

It remains to verify the condition of Lemma \ref{lemma: vertex maps into completed cores}. We isolate the degree-one case, and the higher-dimensional condition follows by considering products of \(1\)-forms.

Let \([a,b]\in T_1\). Since \(\omega_{[a,b]}=\delta_a d \delta_b,\) we have \(\phi(\omega_{[a,b]}) = \phi(\delta_a) d\phi(\delta_b). \) By the above formula for \(\phi^0\), 
\begin{equation}\label{equation: edge pullback formula}
\phi(\omega_{[a,b]}) = \sum_{\substack{[x,y]\in S_1, \\ g_0(x)=a,\ g_0(y)=b}}
\omega_{[x,y]}. \tag{1}
\end{equation} 
Where we used that, for \(x\neq  y\),
\[ \delta_xd_S\delta_y = \begin{cases}
\omega_{[x,y]},&[x,y]\in S_1,\\
0,&[x,y]\notin S_1. \end{cases} \]

We now prove that every source edge maps to a genuine target edge. Let \([x,y]\in S_1\), and set \( a:=g_0(x), \ b:=g_0(y) .\) If \(a=b\), then \([a,a]\) compresses to the vertex \([a]\in T_0\), so the extension condition is automatically satisfied. We may therefore assume \(a\neq  b\). In this case, we claim that \([a,b]\in T _1\).

For \(f\in\Omega_S^0\) and \(x\neq  y\), differentiating \(f\delta_y=f(y)\delta_y\) and multiplying on the left by \(\delta_x\) gives \begin{equation}\label{eq:useful-identity-in-fullness-proof} \delta_x(d_Sf)\delta_y = \bigl(f(y)-f(x)\bigr)\delta_xd_S\delta_y. \end{equation} Since \(\delta_x\phi(\delta_a)=\delta_x,\ \phi(\delta_b)(x)=0, \ \phi(\delta_b)(y)=1_\Bbbk,\) equation \eqref{eq:useful-identity-in-fullness-proof} gives
\[ \begin{aligned}
0\neq \omega_{[x,y]} 
&=\delta_xd_S\phi(\delta_b)\delta_y=\delta_x\phi(d_T\delta_b)\delta_y\\
&=\delta_x\phi(\delta_a)\phi(d_T\delta_b)\delta_y=\delta_x\phi(\delta_a d_T\delta_b)\delta_y.
\end{aligned} \]
Consequently, \(\delta_a d_T\delta_b\neq 0\). Since \(a\neq  b\), this guarantees that \([a,b]\in T_1\).

We now verify the higher-dimensional condition of Lemma \ref{lemma: vertex maps into completed cores}. Let \(\sigma=[v_0,\ldots,v_p]\in S_p\) and denote by \((w_0,\ldots,w_q)\) the compression of \(( g_0(v_0),\ldots,g_0(v_p)),\) obtained by retaining one entry from each maximal constant block. We claim that \([w_0,\ldots,w_q]\in T_q.\)

If \(q=0\), there is nothing to prove. If \(q=1\), the claim follows from the edge case. Suppose that \(q\ge2\). Choose indices \(0\leq i_0<i_1<\ldots<i_q\leq p\) with \(i_r\) belonging to the \(r\)-th maximal constant block, so that \(g_0(v_{i_r})=w_r.\) Let \(\iota\colon[q]\hookrightarrow[p]\) be the order-preserving injection with \(\iota(r)=i_r\). Then \[ \sigma' := S(\iota)(\sigma) =[v_{i_0},\ldots,v_{i_q}] \in S_q. \] For every \(0\leq r<q\), the edge \([v_{i_r},v_{i_{r+1}}]\) is a \(1\)-face of \(\sigma'\). Since \(g_0(v_{i_r}) = w_r, \ g_0 (v_{i_{r+1}}) = w_{r+1}, \ w_r \neq  w_{r+1},\) the edge case gives \( [w_r, w_{r+1}] \in T_1.\)

Consider \( m_T ( w_0, \ldots, w_q) := \omega_{[w_0,w_1]} \wedge \ldots \wedge \omega_{[w_{q-1},w_q]}
\in \Omega_T^q.\) By multiplicativity of \( \phi \) and \eqref{equation: edge pullback formula},
\[ \phi\bigl(m_T(w_0,\ldots,w_q)\bigr) = \bigwedge_{r=0}^{q-1}
\phi\bigl(\omega_{[w_r,w_{r+1}]}\bigr) = \bigwedge_{r=0}^{q-1}
\bigg(
\sum_{\substack{[x,y]\in S_1\\
g_0(x)=w_r,\ g_0(y)=w_{r+1}}}
\omega_{[x,y]}
\bigg).\] For each \( 0 \leq r < q\), the \(r\)-the factor contains \( \omega_{[v_{i_r}, v_{i_{r+1}}]} \) with coefficient \(1_{ \Bbbk }\). Selecting these terms in all \(q\) factors gives \( \omega_{[v_{i_0}, v_{i_1}]} \wedge \ldots \wedge \omega_{[v_{i_{q-1}}, v_{i_q}]} = \rho_{ \sigma' }.\) Since \(S\) has injective vertex maps, no other tuple produces the same basis element \(\rho_{\sigma'}\). Thus \(\rho_{\sigma'}\)
has coefficient \(1_{\Bbbk}\) in \(\phi(m_T)\). In particular, \(\phi(m_T)\neq 0,\)
and \(m_T\neq 0\). By definition, this is equivalent to \([w_0, \ldots, w_q] \in T_q.\)

Lemma \ref{lemma: vertex maps into completed cores} now tells us that \(g_0\) extends uniquely to a morphism \( \widehat g\colon\widehat S\to\widehat T,\ \widehat g([v_0,\ldots,v_p]) = [g_0(v_0) , \ldots, g_0(v_p)].\)

By Remark \ref{remark:unique-lift-simplicialmap}, this morphism is the underlying semisimplicial map of a unique simplicial map \(g\colon L(S)\rightarrow L(T).\) It remains to identify \(g^\ast\) with \(\phi\). By construction, \(\phi\) and \(g^*\) agree in degree \(0\). For every
\([a,b]\in T_1\), equation \eqref{equation: edge pullback formula} is precisely the pullback formula \[ \phi(\omega_{[a,b]}) = \sum_{\substack{[x,y]\in S_1\\ g_0(x)=a,\ g_0(y)=b}} \omega_{[x,y]} = g^*(\omega_{[a,b]}).\] Thus \(\phi\) and \(g^*\) agree in degree \(1\). Since \(\Omega_T^\bullet\) is exterior, it is generated as an algebra by its degree-\(0\) and degree-\(1\) parts. Therefore \(\phi=g^*\) in every degree.

Finally, \(\widehat g\) is unique since its vertex map is uniquely reconstructed from \(\phi^0\), and Lemma \ref{lemma: vertex maps into completed cores} gives at most one extension of that vertex map. Hence \(\Omega^\bullet\) is full.
\end{proof}

\begin{remark}
The use of \(\ssregh_{\leq n}\) is essential. A DGA morphism may identify vertices. Then a nondegenerate source edge may map to a degenerate target edge. Such a map is not a morphism \(S\to T \), but it is a morphism \(\widehat{S}\to \widehat{T}.\)
\end{remark}

Putting together Propositions \ref{proposition: omega functor faithful} and \ref{proposition: omega functor full}, we have that \(\Omega^\bullet: \bigl(\ssregh_{\leq n}\bigr)^{\mathrm{op}}\to \mathsf{Im}_n\) is fully faithful. We now characterize its essential image inside \(\mathsf{DG}_n\), namely those DGAs that arise as \( \Omega_S^\bullet \) for \( \widehat{S} \in \ssregh_{\leq n}\) up to DGA isomorphism. Throughout the remainder of the section, \( A = \Bbbk(X) \), where \( X \) is a finite set, and \( \{ \delta_x \}_{ x \in X}\) denotes the basis of primitive idempotents. For an \( A \)-bimodule \( M \), we write \(M_{x,y}:=\delta_x \cdot M\cdot \delta_y.\) The following is standard.

\begin{lemma}\label{lemma: fibers of FODC are 1 dimensional}
Let \(\Omega^1\) be a first-order differential calculus over \(A\). Then \( \Omega^1_{x,x} = 0, \)  and \(\Omega^1_{x,y} = \Bbbk \cdot (\delta_x d \delta_y) \) for \(x\neq y.\) In particular, \(\dim_\Bbbk \bigl(\Omega^1_{x,y}\bigr) \leq 1,\) when \( (x\neq y).\) \end{lemma} 
\begin{proof}
For \(f,g\in A\), one has \( \delta_x (f dg) \delta_y = f(x)\bigl( g(y) - g(x) \bigr) \delta_x d\delta_y.\) Since \( \Omega^1 \) is spanned by elements \(f dg\), the statement follows. For \(x=y\), the coefficient \(g(y)-g(x)\) vanishes.
\end{proof}

As in Example \ref{example: graph differential calculus}, this tells us that every FODC \( \Omega \) over \( A \) determines a simple directed graph with edges given by \begin{equation}\label{eq:simple-directed-graph-from-fodc} \overline{E}_\Omega := \bigl\{ (x,y) \in X^2\mid x\neq y, \ \Omega^1_{x,y} \neq 0 \bigr\}.\end{equation} If \(\Omega^\bullet\) is an exterior DGA, we have surjections \( (\Omega^1)^{\otimes_A p}_{x,y} \twoheadrightarrow \Omega^p_{x,y}\), so each fiber \( \Omega^p_{x, y} \) is spanned by wedge monomials along length-\(p\) directed paths from \( x \) to \( y \) in \(\overline{E}_\Omega\). Exteriority alone, however, does not force the underlying path of a nonzero monomial to close up to a clique. This is guaranteed by condition \textup{(D2)} below, while condition \textup{(D3)} specifies which of these cliques are actually realized by nonzero monomials and requires those to form a basis (these statements are made precise in the proof of Proposition \ref{proposition: face closure of nonzero monomials}).

\begin{definition}\label{def: concrete DGAs essential image}
Let \(\widehat{\mathsf{DG}}_n\) be the full subcategory of \(\mathsf{DG}_n\) whose objects \((\Omega^\bullet,\wedge,d)\) satisfy the following conditions.

\begin{enumerate}[label=\textbf{(D\arabic*)},ref=D\arabic*,nosep,labelsep=.5em,align=left]
\item There exist a finite set \(X\) and an algebra isomorphism
\(\iota\colon\Bbbk(X)\cong\Omega^0\) such that, writing again \(\delta_x\) for \(\iota(\delta_x)\), the DGA \(\Omega^\bullet\) is exterior.

\item For \(p>1\) and \(x,y\in X\),
\( \Omega^p_{x,y}=0 \ \text{whenever}\ \Omega^1_{x,y}=0.\)
\item For \(x,y\in X\), put
\[ \omega_{[x,y]}:=
\begin{cases}
\delta_xd(\delta_y),&x\neq y,\\ 
0,&x=y.
\end{cases}
\]
For \(1 \leq p \leq n\), define \[ P_p(x,y) := \Bigl\{ (v_0, \ldots, v_p) \ \Big|\ v_0 = x,\ v_p = y,\  (v_i,v_{i+1})\in \overline{E}_\Omega \text{ for } 0 \leq i < p \Bigr\}.\] For \(v=(v_0,\ldots,v_p)\in P_p(x,y)\), put \( m(v):= \omega_{[v_0,v_1]}\wedge\ldots\wedge \omega_{[v_{p-1},v_p]} \in \Omega^p_{x,y}, \) and \[ S_p(x,y) := \{ v \in P_p(x,y) \mid m(v)\neq 0\}.\] Then \((m(v))_{ v\in S_p(x,y)}\) is a \(\Bbbk\)-basis of \(\Omega^p_{x,y}\).
\end{enumerate}
\end{definition}
By construction, for \(\omega_{[x,y]}\) as in {(D3)}, \( \omega_{[x,y]} \in \Omega^1_{x,y}\), and, for \(x\neq y\), Lemma \ref{lemma: fibers of FODC are 1 dimensional} implies that \(\omega_{[x,y]} \neq 0 \) if and only if \( (x,y) \in\overline E_\Omega\). Similar as to how we argued in Example \ref{example: graph differential calculus} to show that the differential is canonical, one can verify that the following holds.

\begin{lemma}\label{lemma: incidence basis}
Let \(\Omega^\bullet\in\widehat{\mathsf{DG}}_n\) with \((X,\iota)\), and \( \omega_{[x,y]} := \delta_x d\delta_y\) as in Definition \ref{def: concrete DGAs essential image}. Then \( d \delta_x = \sum_{u\in X} \bigl(\omega_{[u,x]}-\omega_{[x,u]} \bigr)\) and, for \(x \neq y,\) \( d(\omega_{[x,y]}) = \sum_{u\in X} \bigl( \omega_{[u,x]} \wedge \omega_{[x,y]} -  \omega_{[x,u]}\wedge \omega_{[u,y]} + \omega_{[x,y]}\wedge \omega_{[y,u]} \bigr).\) \end{lemma}

For \(p\geq 1\), extend the notation \(m(v)\) to arbitrary words \( v = (v_0, \ldots, v_p)\in X^{p+1} \) by  the same formula. Since \(\omega_{[x,y]}=0\) whenever \(x=y\) or \((x,y) \notin \overline{E}_\Omega\), one has \( m(v) = 0 \) unless \( v \) is a directed path in \( \overline{E}_\Omega \). This allows to express the differential uniformly as  \[ d m(v_0,\ldots,v_p) = \sum_{j=0}^{p+1}(-1)^j \sum_{w\in X}   m(v_0, \ldots, v_{j-1}, w, v_j, \ldots, v_p).\] The following statement is the algebraic counterpart of face-closure, and its proof also shows that every nonzero monomial is supported on a directed clique.

\begin{proposition}\label{proposition: face closure of nonzero monomials} Let \( \Omega^\bullet \in \widehat{\mathsf{DG}}_n\) with \((X,\iota)\), and \( \omega_{[x,y]}\) as above. If \(p \ge 2\) and \(   m(v_0, \ldots, v_p)\neq 0, \) then \(  m(v_0, \ldots, \widehat v_j, \ldots, v_p)\neq 0 \) for every \(0\leq j\leq p\).
\end{proposition}

\begin{proof} We first observe two properties of nonzero monomials. Suppose that \(  m(v_0,\ldots,v_p)\neq 0.\) For every \(i<j\), the submonomial \( m(v_i, \ldots, v_j) \in\Omega^{j-i}_{v_i,v_j}\) is nonzero, since otherwise associativity would force \( m(v_0,\ldots,v_p)=0\). 

Moreover, the vertices \(v_0,\ldots,v_p\) are distinct. Indeed, suppose that \(v_i=v_j=x\) for some \(i<j\). By
associativity, \(  m(v_i,\ldots,v_j) \in \Omega^{j-i}_{x,x}\) must be nonzero, for otherwise \(  m(v_0,\ldots,v_p)\) would vanish. If \(j-i=1\), this contradicts \(\Omega^1_{x,x}=0\). If \(j-i>1\), condition \textup{(D2)} and the equality \(\Omega^1_{x,x}=0\) imply that \(\Omega^{j-i}_{x,x}=0\), again a contradiction.

Now apply these observations to \( m(v_0, \ldots, v_p).\) For every \(i < j\), the submonomial \( m(v_i, \ldots , v_j) \in \Omega^{j-i}_{v_i, v_j}\) is nonzero. If \( j = i+1\), then \((v_i, v_j)\in \overline{E}_\Omega\) by definition. If \(j-i > 1\), condition \textup{(D2)} implies that \( \Omega^1_{v_i, v_j} \neq 0 \), and hence again \( (v_i,v_j) \in \overline{E}_\Omega\). So every directed diagonal \(v_i \to v_j\), for \(i<j\), exists.

For \(j=0\) or \(j=p\), the face-closure claim is immediate. Indeed, if either \(  m(v_1,\ldots,v_p)=0\) or \(  m(v_0,\ldots,v_{p-1})=0\), associativity forces \(  m(v_0,\ldots,v_p)=0\).

Now fix \(1\leq j\leq p-1\) and put \(\tau=(u_0,\ldots,u_{p-1}):=(v_0,\ldots,\widehat v_j,\ldots,v_p)\). Since \( \omega_{[v_{j-1},v_j]}\wedge \omega_{[v_j,v_{j+1}]}\neq 0\), we have \(\Omega^2_{v_{j-1},v_{j+1}}\neq 0\), and \textup{(D2)} gives
\(\Omega^1_{v_{j-1},v_{j+1}}\neq 0\). Hence \(\tau\) is a directed path word. Assume, for contradiction, that \(  m(\tau)=0\). Projecting the incidence formula for \(d  m(\tau)\) to the \((v_0,v_p)\)-fiber gives \[ 0  = \sum_{r=1}^{p-1}(-1)^r \sum_{w\in X}   m(u_0,\ldots,u_{r-1},w,u_r,\ldots,u_{p-1}), \]
where the boundary contributions \( r = 0 \) and \(r = p \) vanish (non-zero projection forces \( w = v_0\) or \(w = v_p\) respectively, but the resulting monomial then has a repeated vertex and vanishes). Every nonzero term remaining in the sum is again supported on a tuple of pairwise distinct vertices. Exactly one of these terms reproduces the ordered tuple \((v_0,\ldots,v_p)\), namely the insertion at the position vacated by \(v_j\), with \(w=v_j\). It contributes \((-1)^j  m(v_0, \ldots, v_p) \neq 0.\) Every other nonzero term corresponds to a different ordered tuple and is therefore linearly independent of \(  m(v_0,\ldots,v_p)\) by \textup{(D3)}. Hence this contribution cannot cancel, giving a contradiction.
\end{proof}

\begin{theorem}\label{theorem: essential image characterization}
The essential image of \( \Omega^\bullet(-)\colon \bigl(\ssregh_{\leq n}\bigr)^{\mathrm{op}} \rightarrow \mathsf{DG}_n \) is precisely \(\widehat{\mathsf{DG}}_n\). In particular, \(\Omega^\bullet(-):\bigl(\ssregh_{\leq n}\bigr)^{\mathrm{op}} \rightarrow
\widehat{\mathsf{DG}}_n\) is an equivalence of categories.
\end{theorem}

\begin{proof} Let \(S \in \ssreg_{\leq n}\) with \(\widehat S=UL(S)\in\ssregh_{\leq n}.\) Then \(\Omega^\bullet(\widehat S)=\Omega_S^\bullet\) satisfies conditions \textup{(D1)}--\textup{(D3)}, with \(X=S_0\) and \( \Bbbk(S_0) \cong \Omega_S^0.\)

Conversely, let \(\Omega^\bullet\in\widehat{\mathsf{DG}}_n\) with \((X,\iota)\) as in \textup{(D1)}. Define \( S \) by \[  S_0 := X, \qquad S_p := \left\{ (v_0,\ldots,v_p)\in X^{p+1}
\ \middle|\ m(v_0,\ldots,v_p)\neq 0
\right \} \quad (1\leq p\leq n),\] with face maps given by deleting vertices.

By Proposition \ref{proposition: face closure of nonzero monomials}, every face of a simplex lies again in \(S\). Hence \(S\) is a semisimplicial set. Its vertex maps are injective by construction, and its positive-dimensional simplices have distinct vertices by first part of the proof in that same proposition. Thus \(S \in \ssreg_{\leq n}\) and set \(\widehat S:=UL(S) \in \ssregh_{\leq n}.\)  

Define \( \Phi:\Omega^\bullet(\widehat S) = \Omega^\bullet_{S} \rightarrow \Omega^\bullet\) in degree zero by \(\Phi^0 := \iota \) and on higher-degree basis elements by \[\Phi \bigl( \rho_{[v_0,\ldots,v_p]}\bigr) = m\bigl(v_0,\ldots,v_p\bigr). \]
By \textup{(D3)}, this is a graded \(A\)-bimodule isomorphism. It is compatible
with the wedge product by construction, whereas Lemma \ref{lemma: incidence basis} gives compatibility with the differential. Therefore \(\Phi\) is a DGA isomorphism. 
\end{proof}

\section{Cellular sheaves and differential graded modules}\label{section: cellular sheaves, differential modules} Throughout this section, we fix \(S\in\ssreg_{\leq n}\) with \(\widehat S:=UL(S)\in\ssregh_{\leq n}\). We want to characterize the DG-modules over \(\Omega^\bullet (\widehat S) = \Omega^\bullet_S\) that correspond to cellular sheaves on \(S\). Strictly speaking, the DGA \( \Omega^\bullet_S\) is attached to \(\widehat S\), whereas the cellular data considered below are indexed by the simplices of \(S\). Since the base is fixed here, as the passage to \(\widehat S\) plays no role beyond the functoriality of \(\Omega^\bullet\), we work with \(S\) throughout. This section is organized as follows. First, we attach to every $F\in\mathsf{Sh}(S)$ a certain DG-module $M_F$ and study its properties. Afterwards, we define the objects (Definition \ref{definition: cellular module}) and the morphisms of the category $\prescript{\mathsf{c}}{S}{\mathcal{M}}^{\mathsf{DG}}$ (Definition \ref{def:cellularmorphisms}), which we call the category of cellular differential graded modules on $S$. Finally, Theorem \ref{theorem: equivalence between cellular sheaves and differential graded module} shows that this category is equivalent to $\mathsf{Sh}(S)$. We begin with some preliminary observations. For a left \(\Omega^\bullet_S\)-module \(M\), we write \(\mu_M \colon \Omega^\bullet_S \otimes_{A_S} M\rightarrow M,\ \alpha\otimes x\rightarrow \alpha \cdot x\) for the action map, and set \( \mu_\sigma := \mu_M ( \rho_\sigma\otimes-)\colon M\to M\) for \(\sigma\in S_p\), \(1\leq p\leq n\).

\begin{lemma}\label{lemma: tail-fiber decomposition}
For any left \(A_S\)-module \(M\), there is a \(\Bbbk\)-linear isomorphism \[\Omega_S^{p}\otimes_{A_S} M \cong \bigoplus_{\sigma\in S_p} \delta_{t(\sigma)} M, \quad \rho_\sigma \otimes m \rightarrow\delta_{t(\sigma)}m.\] 

\end{lemma}
\begin{proof}
As right \(A_S\)-modules, \(\Omega_S^{p} \cong \bigoplus_{\sigma\in S_p} \delta_{t(\sigma)} A_S\).
For each \(\delta_x \in A_{S}\), the map \(\psi:(\delta_xA_S)\otimes_{A_S} M \rightarrow \delta_xM, \ \psi((\delta_xa)\otimes m) = \delta_x(am)\) is an isomorphism with inverse \(\phi(\delta_x m) = \delta_x\otimes m\), which is well-defined by idempotency. Summing over \(\sigma\) gives the result. 
\end{proof}

\begin{remark}\label{remark: semisimple decomposition of A_s modules}
By semisimplicity of \(A_S \cong \prod_{v\in S_0}\Bbbk\delta_v,\) every left \(A_S\)-module \(M\) decomposes as \(M=\bigoplus_{v\in S_0} M_v,\) where \( M_v \coloneq \delta_v M.\) Here \(\delta_v\) acts on \(M\) by the idempotent endomorphism \(m\to \delta_vm\). Since \(\delta_v\delta_w = \delta_{v,w}\delta_v\), we have \(\delta_v|_{M_v} = \mathrm{id}_{M_v}\) and \( \delta_v|_{M_w} = 0\) for \(v \neq  w\). For \(M\) to be finitely generated as \(A_S\)-module is equivalent to finite-dimensionality of each \(M_v\) as \(\Bbbk\)-vector space. Given the decomposition $M = \bigoplus_{v\in S_0}M_v$, we write \(E_{u,v} \coloneq \delta_u E \delta_v : M_v\to M_u,\) for the $(v,u)$-block of any $E\in  \text{End}_{\Bbbk}(M)$.\end{remark}

\begin{excerpt}\label{excerpt: general remarks on the module structure for a left Omega_S module}
Let \( \sigma = [v_0, \ldots, v_p]\in S_p\). Since \(\rho_\sigma = \rho_\sigma \delta_{v_p}\), there holds \(\rho_{\sigma}\cdot x=\rho_{\sigma}\cdot( \delta_{v_p}x)\) and \[\mu_{\sigma}\big\lvert_{M_u}=0\quad \text{whenever}\quad u\neq  v_p.\]
Moreover, \( \delta_u \cdot \mu_\sigma(x) = (\delta_u \wedge \rho_\sigma)\cdot x = \delta_{u,v_0} \mu_\sigma (x).\) Therefore \(\mu_\sigma ( M_{v_p} ) \subseteq M_{v_0},\) and each \(\rho_{\sigma}\) induces a block map \[ \mu_{\sigma}: M_{v_p}\rightarrow M_{v_0}\] while vanishing on other summands. By exteriority, \( \rho_\sigma = \rho_{[v_0,v_1]}\wedge\ldots\wedge\rho_{[v_{p-1},v_p]}\), so
\begin{equation}\label{equation: composability of action as composition of edge actions}
\mu_\sigma = \mu_{[v_0,v_1]} \circ \ldots \circ \mu_{[v_{p-1},v_p]} : M_{v_p}\to M_{v_0}.
\end{equation}
A left \(\Omega_S^\bullet\)-module structure on \(M\) is therefore completely determined by the family of edge operators \( \{\mu_e: M_{d_0(e)} \to M_{d_1(e)} \}_{ e\in S_1 }\), with the action of higher simplex generators recovered by \eqref{equation: composability of action as composition of edge actions}. Conversely, a family of edge maps \(\{\mu_e:M_{d_0(e)}\to M_{d_1(e)}\}_{e\in S_1}\) defines a unique left \(\Omega_S^\bullet\)-action provided that, whenever \(\rho_{[v_0,v_1]}\wedge\ldots\wedge\rho_{[v_{p-1},v_p]} = 0\) in \(\Omega_S^\bullet\), the corresponding composite \(\mu_{[v_0,v_1]}\circ\ldots\circ\mu_{[v_{p-1},v_p]}\) vanishes.
\end{excerpt}

For $S\in \ssreg_{\leq n}$ with associated DGA $\Omega_S^\bullet$, we write ${}_S\mathcal{M}^{\mathsf{DG}}$ for the category whose objects are left DG-modules $M$ over $\Omega_S^\bullet $ and whose morphisms are $\Omega_S^\bullet$-linear DG-maps $\phi:M\rightarrow N$ of degree $0$.

\begin{excerpt}\label{excerpt: differential graded module associated to a sheaf}
For \(F\in  \mathsf{Sh}(S)\), we now define a left DG-module over $\Omega^\bullet_S$ from this datum. Define the graded vector spaces \[M_{F}^q \coloneqq \bigoplus_{\tau\in S_q} F(\tau), \qquad M_F \coloneqq \bigoplus_{q\ge 0} M_{F}^q.\]
For \(\tau\in S_q\), \(x\in F(\tau)\), the \(\Omega^\bullet_{S}\)-action of a degree-\(p\) basis element \(\rho_\sigma\) is \[ \rho_\sigma\cdot x \coloneq  
\begin{cases}
F(\tau\rightarrow \sigma\star \tau)x \in F(\sigma\star \tau) \subseteq M_F^{p+q}, & \text{if }\sigma\star \tau\in S_{p+q},\\
0, & \text{otherwise.}
\end{cases}\]
For basis elements \(\rho_\sigma,\rho_\eta\) and \(x\in F(\tau)\), functoriality of \(F\) and the multiplication rule in \(\Omega_S^\bullet\) give that \( \rho_\sigma\cdot(\rho_\eta\cdot x) = (\rho_\sigma\wedge\rho_\eta)\cdot x.\) The equality also holds when one of the relevant concatenations is not defined, since both sides then vanish. Moreover, \( 1_{A_S}=\sum_{v\in S_0}\delta_v \) acts as the identity. Hence the action is associative and unital. The differential \(d_{M_F}\) is the signed coboundary \[ d_{M_F}(x) = \sum_{i=0}^{q+1}\ \sum_{\substack{\zeta\in S_{q+1}\\ d_i\zeta=\tau }} (-1)^i F(\tau\rightarrow \zeta) x.\]
\end{excerpt}

\begin{lemma}\label{lemma: differential graded module associated to a cellular sheaf satisfies the Leibniz rule}
In the notation of \ref{excerpt: differential graded module associated to a sheaf}, the internal differential \(d_{M_F}\) satisfies 
\[d_{M_F}^2=0, \quad \text{and}\quad 
d_{M_F}\big(\rho_\sigma\cdot x\big)
=\big(d\rho_\sigma\big)\cdot x + (-1)^p \rho_\sigma\cdot d_{M_F}(x),
\] for homogeneous $\rho_\sigma\in \Omega^{p}$ and $x\in M_F^q$.
\end{lemma}
\begin{proof}
The identity \(d_{M_F}^2=0\) follows from the semisimplicial identities, as the two codimension-two coface contributions associated with \(d_id_j=d_{j-1}d_i\), for \(i<j\), occur with opposite signs.

For the Leibniz rule, let $x\in F(\tau)$ for $\tau\in S_q$. We first discuss the case $\xi \coloneq \sigma\star \tau\in S_{p+q}$. By functoriality, 
\begin{equation}\label{equation: expression for the differential}
d_{M_F}\big(\rho_\sigma\cdot x\big) = d_{M_F}\big(F(\tau\rightarrow \xi)x\big) = \sum_{m=0}^{p+q+1}\sum_{\substack{\zeta\in S_{p+q+1}\\ d_m\zeta=\xi }}\ (-1)^m \ F(\tau\rightarrow \zeta)x.
\end{equation}
We split the $m$-sum into ($0\leq m\leq p$) and ($p+1\leq m\leq p+q+1$) parts. For each pair $(m,\zeta)$ with $m\leq p$ and $d_m\zeta=\xi$, there exists a unique $\theta\in S_{p+1}$ with $d_m\theta=\sigma$ and $\zeta=\theta\star \tau$. So the front contribution becomes
\[\sum_{j=0}^p (-1)^j \sum_{\substack{\theta\in S_{p+1}, \\ d_j\theta=\sigma}}
F(\tau\rightarrow \theta\star \tau)  x .\] Now, use that \(\rho_\theta \cdot x= F(\tau \rightarrow \theta \star \tau)x\) when \( \theta \star \tau\) is defined and \(0\) otherwise.

For \(j=p+1\), every \(\theta\in S_{p+1}\) satisfying \(d_{p+1}\theta=\sigma\) has vertex
word \([v_0,\ldots,v_p,u].\)
Since \(\theta\) has pairwise distinct vertices, \(u\neq  v_p\).
On the other hand, \(s(\tau)=v_p\), because
\(\sigma\star\tau\) is defined. Hence \(t(\theta)=u\neq  s(\tau)\), so \(\theta\star\tau\) is not defined and we can extend the above sum \[ \sum_{j=0}^p (-1)^j \sum_{\substack{ \theta \in S_{p+1}, \\ d_j\theta=\sigma}}
F(\tau\rightarrow \theta\star \tau)  x =
\Bigg(\sum_{j=0}^{p+1} (-1)^j \sum_{\substack{\theta\in S_{p+1},\\ d_j\theta=\sigma}} \rho_\theta\Bigg)\cdot x
= (d\rho_\sigma)\cdot x.
\]\smallskip\noindent
Likewise for the tail part, write $i=p+\ell$ with $1\leq \ell\leq q + 1$. Then $d_i\zeta=\xi$ if and only if \(\zeta = \sigma\star \eta\) with \( d_\ell \eta = \tau\). So the tail contribution becomes
\[ \sum_{\ell =1 }^{q+1}(-1)^{p+\ell} \sum_{\substack { \eta \in S_{q+1}\\ d_\ell \eta= \tau}}  F(\tau \rightarrow \sigma \star \eta)x
=(-1)^p \rho_\sigma \cdot
\Big(\sum_{\ell =0 }^{q+1} \sum_{\substack{\eta \in S_{q+1},\\ d_\ell \eta= \tau}} (-1)^\ell F(\tau\rightarrow \eta)\Big) x = (-1)^p \rho_\sigma\cdot d_{M_F}(x), \] where we again argue by distinctness of vertices to include the \(\ell=0\) term. Adding front and tail proves the statement.

It remains to consider the case in which \(\sigma\star\tau\) is not defined. Then \( \rho_\sigma \cdot x=0\), so the left-hand side of the Leibniz identity vanishes. 

In the expansion of \(( d \rho_\sigma) \cdot x\), a term indexed by
\( d_j \theta = \sigma\) with \(j \leq p\) can be nonzero only if \( \theta \star \tau\) is defined. But then taking its \(j\)-th face would imply that \(\sigma \star \tau\) is defined, a contradiction. Similarly, in the expansion of \((-1)^p \rho_\sigma \cdot d_{M_F}x\), a term indexed by \(d_\ell\eta=\tau\) with \(\ell\ge1\) can be nonzero only if
\(\sigma\star\eta\) is defined, and taking its \((p+\ell)\)-th face again produces \(\sigma\star\tau\), a contradiction.

Thus only the terms \(j=p+1\) and \(\ell=0\) can remain. These terms are indexed by the same simplices \[ \zeta=\theta\star\tau=\sigma\star\eta \]
and both have value \(F(\tau\to\zeta)x\). Their signs are respectively \( (-1)^{p+1}\) and \((-1)^p\), so they cancel. 
\end{proof}

Combining \ref{excerpt: differential graded module associated to a sheaf} and \ref{lemma: differential graded module associated to a cellular sheaf satisfies the Leibniz rule} we obtain that $M_F$ is indeed a left DG-module over $\Omega^\bullet_S$.

\begin{remark}\label{remark:DG-modules-over-S} The categories \({}_S\mathcal{M}^{\mathsf{DG}}\) can also be  considered over varying bases. Namely, let $\mathcal{M}^{\mathsf{DG}}$ be the category whose objects are pairs $(M,S)$ with $M\in {}_S\mathcal{M}^{\mathsf{DG}}$ and morphisms $(M,S)\rightarrow (N,T)$ are pairs $(f,\phi)$ where $f:S\rightarrow T$ is a semisimplicial map and $\phi:M\rightarrow N$ is a degree zero chain morphism such that \[\phi (\Omega({f})(\alpha)\cdot m)=\alpha \cdot \phi(m), \quad \alpha \in \Omega_T^\bullet,\ m \in M.\]
Composition is given by $(g,\psi)\circ(f,\phi)=(g\circ f,\ \psi\circ \phi)$. The projection $\pi:\mathcal{M}^{\mathsf{DG}}\rightarrow \ssreg_{\leq n}$ is a Grothendieck opfibration (see \cite{Vistoli2005Descent}) whose fiber over $S$ is ${}_S\mathcal{M}^{\mathsf{DG}}$. In particular, morphisms over $\mathrm{id}_S$ are the $\Omega_S^\bullet$-linear DG-maps.
\end{remark}

\begin{excerpt}\label{excerpt: extending the assignment of cellular sheaves to differential graded algebras to morphisms}
For \(F \in \mathsf{Sh}(S),\) write \(\mathcal{R}(F)=M_F\) and for \(\eta\in \operatorname{Hom}_{\mathsf{Sh}(S)}(F,F')\), define \(\mathcal{R}(\eta):M_F\rightarrow M_{F'}\) by the block map \[\mathcal{R}(\eta) \big \lvert_ {M_F^q} \coloneq \bigoplus_{\tau \in S_q} \eta_\tau.\] 
\end{excerpt}

\begin{lemma}\label{lemma: assignments of cellular sheaves to differential graded algebras is functorial}
For fixed $S\in \ssreg_{\leq n}$, the assignments of \ref{excerpt: differential graded module associated to a sheaf} and \ref{excerpt: extending the assignment of cellular sheaves to differential graded algebras to morphisms} define a functor \(\mathcal{R}: \mathsf{Sh}(S) \rightarrow {}_S\mathcal{M}^{\mathsf{DG}}\).
\end{lemma}

\begin{proof}
By Lemma \ref{lemma: differential graded module associated to a cellular sheaf satisfies the Leibniz rule}, \(M_F\) is a DG-module over \(\Omega^\bullet_S\). For \(F,F',\eta\) as in \ref{excerpt: extending the assignment of cellular sheaves to differential graded algebras to morphisms}, \(\mathcal{R}(\eta)\) is $\Bbbk$-linear and degree \(0\) by definition. It remains to verify compatibility with the action and the differential. Let \(\tau\in S_q,\ \sigma \in S_p, \ x\in F(\tau)\) with \(\sigma\star\tau \in S_{p+q}\). Naturality of \(\eta\) gives 
\[\mathcal{R}(\eta)\left(\rho_\sigma\cdot x\right)
=\mathcal{R}(\eta)\left(F(\tau\rightarrow \sigma\star\tau)x\right)
=\eta_{\sigma\star \tau}F(\tau\rightarrow \sigma\star\tau)x
=F'(\tau\rightarrow \sigma\star\tau)\eta_\tau(x)
=\rho_\sigma\cdot \mathcal{R}(\eta)(x).\]
Similarly, \[ \begin{split}
\mathcal{R}(\eta)\big(d_{M_F}x\big)
&=\sum_{i=0}^{q+1} \sum_{\substack{\zeta \in S_{q+1}\\ d_i \zeta=\tau}}(-1)^i\eta_\zeta\bigl(F(\tau\rightarrow \zeta)x\bigr)
=\sum_{i=0}^{q+1} \sum_{\substack{\zeta \in S_{q+1}\\ d_i \zeta=\tau}}(-1)^i F'(\tau\rightarrow \zeta)\eta_\tau(x)
= d_{M_{F'}}\big(\mathcal{R}(\eta)x\big).
\end{split}
\] By construction, \(\mathcal{R}(\mathrm{id}_F)=\mathrm{id}_{M_F}\) holds stalkwise and for $\eta':F'\Rightarrow F''$, also
$\mathcal{R}(\eta'\circ \eta)\big\lvert _{F(\tau)}=(\eta'_\tau\circ \eta_\tau)=
\mathcal{R}(\eta')\big\lvert _{F'(\tau)}\circ \mathcal{R}(\eta)\big\lvert _{F(\tau)}$.
\end{proof}

To characterize the essential image of \(\mathcal{R}\), we introduce the following terminology.

\begin{definition}\label{definition: cellular module}
A module \(M\in {}_S\mathcal{M}^{\mathsf{DG}}\) is said to be \emph{cellular} if it satisfies the following additional conditions:
\begin{enumerate}[label=\textbf{(S\arabic*)},ref=S\arabic*,nosep,labelsep=.5em,align=left]
\item\label{ax:S1prime} For each $ 0 \leq q \leq n $, there is a decomposition \( M^q=\bigoplus_{\tau\in S_q} M(\tau)\), where each $M(\tau)$ is finite-dimensional over \(\Bbbk\), and, for $\tau= [v_0,\ldots,v_q],$ \( M(\tau)\subseteq M_{v_0}=M_{s(\tau)}.\) Moreover, \(M^q=0\) for \(q>n\). We write \[ \iota_\tau : M(\tau) \hookrightarrow M^q, \qquad
\pi_\tau : M^q\twoheadrightarrow M(\tau) \] for the natural inclusions and projections, respectively. 
\item\label{ax:S2} For each $\tau\in S_q$, $\zeta\in S_{q+1}$, let \(D_{\tau\rightarrow\zeta} \coloneq \pi_\zeta\circ d_M\circ \iota_\tau.\) Then
\[d_M\big\lvert _{M(\tau)}=\sum_{i=0}^{q+1} \sum_{\substack{\zeta\in S_{q+1}\\ d_i\zeta=\tau }}
D_{\tau\rightarrow \zeta}.\]
\end{enumerate}
\end{definition}

Conditions \textup{(S1)}--\textup{(S2)} are natural given \ref{excerpt: differential graded module associated to a sheaf}. The former refines the grading into simplex stalks such that \( M( \tau)\subseteq M_{s(\tau)}\) (cf. \ref{excerpt: general remarks on the module structure for a left Omega_S module}) while the latter restricts the support of the differential to codimension-one faces. What appears missing is any explicit compatibility between the \(\Omega^\bullet_S\)-action and \(d_M\). The next lemma shows that no such condition is needed. Under \textup{(S1)} and \textup{(S2)}, the graded Leibniz rule alone forces the edge action to coincide with the front-face component of the differential.

\begin{proposition}\label{proposition:edge-action-from-front-faces}
Let \(M\in {}_S\mathcal{M}^{\mathsf{DG}}\) satisfy \(\textup{(S1)}\) and \(\textup{(S2)}\). For \(e\in S_1\), \(\tau\in S_q\), if \(e\star\tau\) defined, then \[\pi_\zeta \circ ( \omega_e \cdot - ) \circ\iota_\tau =   
\begin{cases} 
 D_{\tau\to e\star\tau} & \text{if }\zeta = e\star\tau,\\
0, & \text{otherwise.}
\end{cases}\] 
If \( e \star \tau \) is not defined, then \(\omega_e\cdot M(\tau)=0\).
\end{proposition}

\begin{proof}
Let \(v:=s(\tau)\) and \(x\in M(\tau)\). Condition \(\textup{(S1)}\) gives \( x = \delta_v x\). Since \(\delta_v\) has degree \(0\), the Leibniz rule gives \(d_M x = d_M (\delta_v x) = d \delta_v\cdot x+\delta_v d_M x,\) hence
\[ d_Mx-\delta_vd_Mx=d\delta_v\cdot x. \tag{1} \] By \(\textup{(S2)}\),
\(d_Mx = \sum_{i=0}^{q+1} \sum_{\substack{\zeta\in S_{q+1}, \ d_i\zeta=\tau}}
D_{\tau\to\zeta}x. \) If \(i\geq1\), then \(s(\zeta)=s(\tau)=v\), so \(D_{\tau\to\zeta}x\in M(\zeta)\subseteq M_v\) and \(\delta_v\) fixes this summand. If \(i=0\), then \( \zeta = e\star\tau\) for an edge \(e\) satisfying \(d_0e=v\). The first vertex of \(\zeta\) is \(d_1e\neq  v\), by vertex-distinctness. Thus \(\delta_v\) annihilates this summand. Therefore, \[ d_Mx-\delta_vd_Mx = \sum_{\substack{\zeta\in S_{q+1}\\ d_0\zeta=\tau}}
D_{\tau\to\zeta}x. \tag{2} \]
On the other hand, \( d \delta_v = \sum_{e\in S_1,\ d_0e=v}\omega_e - \sum_{e\in S_1, \ d_1 e=v}\omega_e.\) Since \(\omega_e = \omega_e\delta_{d_0e}\) and \(x=\delta_v x\), we have \(\omega_e\cdot x=0\) unless \(d_0e=v\). In particular, every term in the second sum vanishes on \(x\) by vertex distinctness. Thus \[ d\delta_v\cdot x = \sum_{\substack{e\in S_1\\ d_0e=v}} \omega_e\cdot x. \tag{3} \]
Combining \( (1) \), \( (2) \), and \( (3) \), we obtain
\( \sum_{\substack{ \zeta \in S_{q+1},\ d_0\zeta = \tau}} D_{ \tau \to \zeta} x = \sum_{\substack{e\in S_1, \ d_0e=v}} \omega_e \cdot x \) and projecting this to \( M( \zeta ) \) gives \[ \sum_{\substack{e\in S_1\\ d_0e=v}} \pi_\zeta(\omega_e\cdot x)=
\begin{cases} D_{\tau\to\zeta}x,
& d_0\zeta=\tau,\\
0, & d_0\zeta\neq\tau.
\end{cases} \tag{4} \]

Now fix \(e\in S_1\). If \(d_0e\neq v\), then \(\omega_e\cdot x=0\), since \( \omega_e = \omega_e \delta_{d_0e}\) and \( x = \delta_v x\). Thus, we may assume that \(d_0 e = v\), and put \( u := d_1 e\).

Since \( \omega_e=\delta_u \omega_e \delta_v,\) the element \( \omega_e \cdot x\) lies in \( M_u\). Hence \( \pi_\zeta(\omega_e\cdot x)=0 \) unless \( s(\zeta) = u.\) If \(s(\zeta) = u\), then \( e \) is the unique edge \( e'\in S_1 \) such that \( d_0 e' = v, \ d_1 e' = s (\zeta).\) Consequently, the sum on the left-hand side of \((4)\) reduces to the single term \(\pi_\zeta(\omega_e\cdot x)\), and we have
\[ \pi_\zeta(\omega_e\cdot x) = \begin{cases}
D_{\tau\to\zeta}x,
& d_0\zeta=\tau,\\
0,
& d_0\zeta\neq\tau.
\end{cases}
\] Under the conditions \( d_0 e = s(\tau)\) and \( s (\zeta) = d_1e\), vertex-injectivity gives \( d_0 \zeta = \tau \) if and only if \(\zeta = e \star \tau. \) Therefore \[ \pi_\zeta( \omega_e \cdot x) = \begin{cases}
D_{\tau \to e \star \tau} x,
& \zeta=e\star\tau, \\
0, & \text{otherwise.}
\end{cases}\]
If \(e\star\tau\) is not defined, \(\omega_e\cdot x=0\).
\end{proof}

To preserve the simplexwise data in \textup{(S1)}, we restrict to DG-maps that are compatible with the summands.

\begin{definition}\label{def:cellularmorphisms} A morphism \(\phi: M\rightarrow N\) in \({}_S\mathcal{M}^{\mathsf{DG}}\) is said to be \emph{cellular} if $\phi(M(\tau))\subseteq N(\tau)$ for all $\tau\in S$. We denote by $\prescript{\mathsf{c}}{S}{\mathcal{M}}^{\mathsf{DG}}
$ the subcategory of ${}_S\mathcal{M}^{\mathsf{DG}}$ given by cellular differential graded modules with cellular morphisms between them.
\end{definition}

\begin{excerpt}\label{excerpt: first introduction of the sheaf associated to a cellular module}
The DG-module \(M_F\) constructed from a cellular sheaf \(F\) is cellular. Indeed, \(M_F^q = \bigoplus_{\tau \in S_q} F(\tau) \) by construction, so the simplexwise decomposition in \(\textup{(S1)}\) is satisfied with \(M_F(\tau)=F(\tau)\). If \(\tau=[v_0,\ldots,v_q]\), then the \(A_S\)-action on \(F(\tau)\) is through \(\delta_{v_0}\), so \(M_F(\tau)\subseteq (M_F)_{v_0}. \) For \(\textup{(S2)}\), if \(d_i\zeta=\tau\), the \((\tau,\zeta)\)-block of \(d_{M_F}\) is \(D_{\tau\to\zeta}=(-1)^iF(\tau\to\zeta).\) Hence \(d_{M_F}\) is supported on codimension-one cofaces. Conversely, given \(M \in  \prescript{\mathsf{c}}{S}{\mathcal{M}}^{\mathsf{DG}}\), one can define a cellular sheaf \(F_M\) by putting \[ F_M(\tau)  \coloneq  M(\tau), \quad F_M ( \tau \rightarrow \zeta) \coloneq (-1)^i D_{\tau\rightarrow \zeta}, \quad \text{for each } \zeta \text{ with }d_i\zeta=\tau, \ \text{for some }i.\] For a general inclusion \(\tau \leq \xi\), set \(F_M(\tau\rightarrow \xi)\) to be the signed composition of the $D_{\bullet\rightarrow\bullet}$ along a maximal chain from $\tau$ to $\xi$ in the Hasse diagram (cf. \ref{definition: face poset, hasse diagram associated to a simplicial set}). The fact that \(d_M^2=0\) guarantees that this is well-defined, as the following lemma shows.
\end{excerpt}

\begin{lemma}\label{lemma: sheaf associated to a cellular differential graded module is functorial and well-defined}
For $M\in \prescript{\mathsf{c}}{S}{\mathcal{M}}^{\mathsf{DG}}$, \ \(F_M: P_S\rightarrow\mathrm{Vect}_\Bbbk\) defines a cellular sheaf on $S$. Additionally, the assignment $M\to F_M$ defines a functor $\mathcal{S}: \prescript{\mathsf{c}}{S}{\mathcal{M}}^{\mathsf{DG}}\rightarrow \mathsf{Sh}(S)$.
\end{lemma}
\begin{proof}
Let \(\xi\in S_{q+2}\) and let $\tau\in S_q$ be a codimension-$2$ face of $\xi$. There is a unique pair $i<j$ with $\tau=d_i d_j\xi=d_{j-1} d_i\xi$, so exactly two codimension-\(1\) intermediates occur. Denote these by \(\eta  \coloneq  d_j\xi  \ \text{and}\ \eta'  \coloneq  d_i\xi\). By (S2), the $(\tau,\xi)$-block of $d_M^2$ is precisely
\(D_{\eta\rightarrow \xi}\circ D_{\tau\rightarrow \eta} + D_{\eta'\rightarrow \xi}\circ D_{\tau\rightarrow \eta'}.\)
Since $d_M^2=0$, this sum vanishes, so 
\begin{equation*}
\begin{split}
F_M(\eta\rightarrow \xi)\circ F_M(\tau\rightarrow \eta) &= (-1)^{i+j}D_{\eta\rightarrow \xi}\circ D_{\tau\rightarrow \eta}
= (-1)^{i+j-1}D_{\eta'\rightarrow \xi}\circ D_{\tau\rightarrow \eta'}
\\
& = F_M(\eta'\rightarrow \xi)\circ F_M(\tau\rightarrow \eta').
\end{split}
\end{equation*}

For a general inclusion \(\tau\leq\xi\), any two maximal chains from \(\tau\) to \(\xi\) are related by a sequence of interchanges of two successive face insertions. Each such interchange replaces one length-\(2\) segment by the other side of a codimension-\(2\) square. The preceding computation therefore shows that the associated signed composites agree. Hence \(F_M(\tau\to\xi)\) is independent of the chosen chain, and the resulting maps are compatible with composition.

It remains to define \(\mathcal S\) on morphisms. Let \(\phi\colon M\to N\) be a cellular DG-map, and set \[ \mathcal S(\phi)_\tau := \phi\big|_{M(\tau)}
\colon M(\tau)\rightarrow N(\tau). \] For \(d_i\zeta=\tau\), cellularity of \(\phi\) and the identity \(\phi d_M=d_N\phi\) give \(\phi_\zeta D^M_{\tau\to\zeta} = D^N_{\tau\to\zeta}\phi_\tau.\) Multiplying by \((-1)^i\) yields \( \phi_\zeta F_M(\tau\to\zeta) = F_N(\tau\to\zeta)\phi_\tau.\) Thus \((\phi_\tau)_\tau\) is natural on codimension-one inclusions, and hence on every inclusion in \(P_S\). Identities and compositions are preserved stalkwise, so \(M\to F_M\) defines a functor.
\end{proof}

\begin{theorem}\label{theorem: equivalence between cellular sheaves and differential graded module}
Let $S\in \ssreg_{\leq n}$. Then the pair \((\mathcal{R}, \mathcal{S})\) defines an equivalence of categories \(\mathsf{Sh}(S) \simeq\prescript{\mathsf{c}}{S}{\mathcal{M}}^{\mathsf{DG}}.\)
\end{theorem}
\begin{proof} We show that \(\mathcal R\dashv\mathcal S\) and that its unit and counit are isomorphisms. Define
\[\Theta:\mathrm{Hom}_{\prescript{\mathsf{c}}{S}{\mathcal{M}}^{\mathsf{DG}}}\big(\mathcal{R}(F),M\big)
 \rightarrow{}
\mathrm{Nat}\big(F,\mathcal{S}(M)\big), \quad
\Theta(\Phi)_\tau \coloneq \Phi\lvert_{F(\tau)}: F(\tau) \rightarrow M(\tau),
\]
and
\[ \Xi:\mathrm{Nat}\big(F,\mathcal{S}(M)\big) \rightarrow\ \mathrm{Hom}_{\prescript{\mathsf{c}}{S}{\mathcal{M}}^{\mathsf{DG}}}\big(\mathcal{R}(F),M\big),\quad 
\Xi(\nu)\lvert_{F(\tau)} \coloneq \nu_\tau: F(\tau) \rightarrow M(\tau), \]
extended linearly for \(M_F= \oplus_{\tau} F(\tau)\). Then $\Xi(\Theta(\Phi))=\Phi$ blockwise and $\Theta(\Xi(\nu))=\nu$ stalkwise.

For \(\tau, \zeta\) with $d_i\zeta=\tau$ for some \(i\), projecting \(\Phi\circ d_{M_F}=d_M\circ \Phi\) to the $(\tau,\zeta)$-block gives \[
\Theta(\Phi)_\zeta\circ ((-1)^iF(\tau\rightarrow\zeta))
 = 
D^M_{\tau\rightarrow\zeta}\circ \Theta(\Phi)_\tau.
\]
Since $F_M(\tau\rightarrow\zeta)=(-1)^iD^M_{\tau\rightarrow\zeta}$, also $F_M(\tau\rightarrow\zeta)\circ \Theta(\Phi)_\tau=\Theta(\Phi)_\zeta\circ F(\tau\rightarrow\zeta)$.
One can again extend to all $\tau\leq\xi$, so $\Theta(\Phi)$ is a natural transformation \(F\rightarrow F_M.\) 

To see that \(\Xi(\nu)\) is $\Omega_S^\bullet$-linear, we first show \(A_S\)-linearity. Since $\Xi(\nu)$ is \(\Bbbk\)-linear and preserves the simplex decompositions, for $f\in A_{S}, \tau=[u_0,\ldots, u_q]\in S_q, \ x\in F(\tau)$, \[\Xi(\nu)(f \cdot x)= \Xi(\nu)(f(u_0)x)=f(u_0)\Xi(\nu)(x)= f \cdot \Xi(\nu)(x).\]
Likewise, for \(e\in S_1\), \(x\in F(\tau)\), assume first that \(\zeta=e \star \tau\) is defined. Then \[ \Xi(\nu) (\omega_e \cdot x) = \Xi(\nu)(F(\tau\to\zeta)x) = \nu_\zeta F(\tau\to\zeta)x = F_M(\tau\to\zeta)\nu_\tau(x). \]
Since \(\zeta=e\star\tau\) is the front coface of \(\tau\), one has
\( F_M ( \tau \to \zeta)=D^M_{\tau \to \zeta}. \) By Proposition \ref{proposition:edge-action-from-front-faces}, \( D^M_{\tau\to\zeta}\nu_\tau(x)=\omega_e\cdot\nu_\tau(x),\) so \( \Xi(\nu)(\omega_e\cdot x) = \omega_e\cdot\Xi(\nu)(x).\) If \(e\star\tau\) is not defined, both sides vanish. Indeed, on \(\mathcal{R}(F)\) this is by construction, and on \(M\) it follows from Proposition \ref{proposition:edge-action-from-front-faces}. For any higher degree $\sigma=[v_0,\ldots,v_p]\in S_p$, we have $ \rho_\sigma = \rho_{e_1} \wedge \ldots \wedge \rho_{e_p}$, where $e_k = [v_{k-1},v_k]$. By associativity of the left action, \[ \Xi(\nu)(\rho_\sigma\cdot x)
=\Xi(\nu)\big(\rho_{e_1}\cdot(\rho_{e_2}\cdot(\ldots(\rho_{e_p}\cdot x)\ldots ))\big)
=\rho_{e_1}\cdot\big(\rho_{e_2}\cdot(\ldots(\rho_{e_p}\cdot \Xi(\nu)(x))\ldots)\big)
=\rho_\sigma\cdot \Xi(\nu)(x).\]
Hence $\Xi(\nu)$ is $\Omega_S^\bullet$-linear.

For compatibility with the differential, observe that
\[ \Xi(\nu)\big(d_{M_F}x\big)
=\sum_{i,\zeta}(-1)^i \nu_\zeta F(\tau\rightarrow\zeta)x
=\sum_{i,\zeta}(-1)^i F_M(\tau\rightarrow\zeta)\nu_\tau(x)
= d_M\big(\Xi(\nu) x\big), 
\]
where the last equality follows from (S2). Thus $\Xi(\nu)$ is a DG-map. It is moreover cellular, since $\Xi(\nu)(M_F(\tau)) = \nu_\tau(F(\tau)) \subseteq M(\tau)$ by construction.

Naturality of \(\Theta\) and \(\Xi\) follow by routine diagram chases. It remains to show that the unit and counit are natural isomorphisms. Since, for $M\in \prescript{\mathsf{c}}{S}{\mathcal{M}}^{\mathsf{DG}}$,
\[ \mathcal{R}(\mathcal{S}(M))^q = \bigoplus_{\tau\in S_q} F_M(\tau) = \bigoplus_{\tau\in S_q} M(\tau)=M^q,\]
$\epsilon_M$ is simply the block identity. It commutes with the differentials, because for \(d_i\zeta=\tau\), \[ D^{\mathcal{R}(\mathcal{S}(M))}_{\tau\to\zeta} = (-1)^iF_M(\tau\to\zeta) = (-1)^i(-1)^iD^M_{\tau\to\zeta} = D^M_{\tau\to\zeta}.\] It also commutes with the \(\Omega_S^\bullet\)-action. It is enough to check edge generators. For \(e\in S_1\) and \(\tau\) with \(e\star\tau\) defined,
\[ (\omega_e)^{\mathcal R(\mathcal S(M))}|_{\tau\to e\star\tau} = F_M(\tau\to e\star\tau) = D^M_{\tau\to e\star\tau} = (\omega_e)^M|_{\tau\to e\star\tau}, \]
where the last equality follows from Proposition \ref{proposition:edge-action-from-front-faces}. If \(e\star\tau\) is not defined, both edge actions vanish by the same proposition. Since both left \(\Omega_S^\bullet\)-module structures are determined by their edge operators, the full \(\Omega_S^\bullet\)-actions agree. Hence \( \epsilon_M \) is a DG-isomorphism, natural in \(M\). 

Lastly, for \(F\in \mathsf{Sh}(S)\), define \(\eta_F:F\to \mathcal{S}(\mathcal{R}(F))\) by
\((\eta_F)_\tau \coloneq \mathrm{id}_{F(\tau)}.\) Whenever \(d_i\zeta = \tau,\)
\[ \mathcal{S}(\mathcal{R}(F))(\tau\rightarrow \zeta) = (-1)^iD^{M_F}_{\tau\rightarrow \zeta} = (-1)^i \pi_\zeta d_{M_F}\iota_\tau=(-1)^i(-1)^i F(\tau\rightarrow \zeta)=F(\tau\rightarrow \zeta), \] so $\eta_F : F\rightarrow \mathcal{S}(\mathcal{R}(F))$ is a natural isomorphism. Since both \( F \) and \(\mathcal{S}(\mathcal{R}(F))\) are functorial on \(P_S\), agreement on codimension-one inclusions extends to all \(\tau\leq\xi\).
\end{proof}

\begin{excerpt}\label{excerpt: relation with the incidence algebra}
Let \( \Bbbk[P_S] \) denote the incidence algebra of the face poset \(P_S\) of
\(S\in \ssreg_{\leq n}\). By Remark \ref{remark: relation between AW DGA and the incidence algebra}, the map \[ \Psi:(\Omega_S^\bullet,\wedge)\to \Bbbk[P_S],\quad \Psi(\rho_\sigma)=T_\sigma,\] is an injective anti-algebra morphism with image \(C\subseteq \Bbbk[P_S]\). There is a standard equivalence between \(\Bbbk\)-linear cellular sheaves on a finite poset
and finite-dimensional right modules over its incidence algebra (see, for instance,
\cite[Lemma 2.7]{Ladkani2008DerivedEquivalencesSheavesFinitePosets}): \[ E: \mathsf{Sh}(S) \to \mathcal{M}_{\Bbbk[P_S]},\quad
F\to E(F) \coloneq \bigoplus_{\tau\in P_S}F(\tau),\] where the right action on $E(F)$ is defined as \( m\cdot e_{\tau,\xi} \coloneq \iota_\xi \bigl(F(\tau\leq \xi)\pi_\tau(m)\bigr),\ ( \tau\leq \xi)\), and for a morphism \(\eta:F\to F'\), we define \( E(\eta) \coloneq \bigoplus_{\tau\in P_S}\eta_\tau.\) Write \( \operatorname{res}_{j}: \mathcal{M}_{\Bbbk[P_S]} \rightarrow \mathcal{M}_C \) for the restriction of scalars along \(j: C \hookrightarrow \Bbbk[P_S]\) and let \(U :{}_S\mathcal{M}^{\mathsf{DG}}\rightarrow \mathcal {M}_C\) be the forgetful functor omitting the grading and differential, with right
\(C\)-action given by \(m\cdot \Psi(\rho_\sigma)\coloneq\rho_\sigma\cdot m.\)
\end{excerpt}

\begin{proposition}
There is a natural isomorphism of functors
\(U\circ \mathcal R \simeq  \operatorname{res}_{j}\circ E.\)
\end{proposition}

\begin{proof}
For each \(F\in \mathsf{Sh}(S)\), both \(U(M_F)\) and \(\operatorname{res}_{j}(E(F))\) share the same underlying vector space \(\bigoplus_{\tau\in P_S}F(\tau).\)
Define \(\vartheta_F: U(M_F) \rightarrow \operatorname{res}_{j}(E(F)) \)
to be the identity on this direct sum. It remains to check \(C\)-linearity.

Let \(\sigma\in P_S\), and \(x\in F(\tau)\subseteq \bigoplus_{\eta\in P_S}F(\eta)\) for some \(\tau\in P_S\). In \(U(M_F)\), the action of \(T_\sigma=\Psi(\rho_\sigma)\) is
\[  x \cdot T_\sigma =\rho_\sigma\cdot x=
\begin{cases}
F(\tau\leq \sigma\star\tau) x, & \sigma\star\tau \text{ defined},\\
0, & \text{otherwise}.
\end{cases}
\]

In \( \operatorname{res}_{j}(E(F))\),
\( T_\sigma = \sum_{\substack{\eta\in P_S,\\  \sigma\star\eta\ \text{defined}}} e_{\eta,\sigma\star\eta}, \)
and only the term \(\eta=\tau\) acts nontrivially on \(x\in F(\tau)\). Hence
\[ x \cdot T_\sigma = \begin{cases}  x \cdot e_{\tau,\sigma\star\tau}
=F(\tau\leq \sigma\star\tau) x, & \sigma\star\tau \text{ defined},\\
0, & \text{otherwise}. \end{cases} \]

Thus \(\vartheta_F\) is \(C\)-linear. For a morphism \(\eta:F\to F'\) of sheaves, both \(U(\mathcal R(\eta))\) and \(E(\eta)\) are the block map \(\bigoplus_{\tau}\eta_\tau\), so \(\{\vartheta_F\}_F\) is natural.
\end{proof}

\begin{excerpt}\label{excerpt: right module version of sheaf module correspondence}
The correspondence between left DG-modules and cellular sheaves established in Theorem \ref{theorem: equivalence between cellular sheaves and differential graded module} has a natural right-handed version, obtained by postpending simplices rather than prepending them. We sketch the analogues of the left definitions and leave details to the reader.

For a right \(A_S\)-module \(M\), write \(M_v:=M\delta_v.\) Then \(M\otimes_{A_S}\Omega_S^p \cong \bigoplus_{\sigma\in S_p} M_{{s(\sigma)}}\) decomposes along source vertices, and a right \(\Omega_S^\bullet\)-module structure is completely determined by edge operators \( \{ \mu_e^{\mathrm r} \colon M_{d_1(e)}\to M_{d_0(e)}\}_{e\in S_1}\). Given \(F\in\mathsf{Sh}(S)\), the graded vector space
\(M_F^{\mathrm r}\coloneqq \bigoplus_{q\ge 0}\bigoplus_{\tau\in S_q}F(\tau)\)
carries the right action
\[
x\cdot\rho_\sigma \coloneqq
\begin{cases}
F(\tau\to \tau\star\sigma)x, & \tau\star\sigma\in S_{p+q},\\
0, & \text{otherwise},
\end{cases} \quad \tau\in S_q,\ x\in F(\tau),\ \rho_\sigma\in\Omega_S^p, \]
together with the same signed coboundary differential as in
\ref{excerpt: differential graded module associated to a sheaf}.
As in Lemma
\ref{lemma: differential graded module associated to a cellular sheaf satisfies the Leibniz rule}, one can show that this defines a right DG-module over \( \Omega_S^\bullet \).
\end{excerpt}

\begin{definition}\label{def:right-cellular-module}
A right DG-module \(M\in\mathcal{M}_S^{\mathrm{DG}}\) is \emph{cellular} if it satisfies \textbf{(S1)} and \textbf{(S2)}
of Definition \ref{definition: cellular module}, with the simplex stalk
condition replaced by \(M(\tau)\subseteq M_{t(\tau)}\) for
\(\tau=[v_0,\ldots,v_q]\), where \(t(\tau)=v_q\).
\end{definition}

\begin{remark}\label{remark: right edge action from terminal faces}
The right-handed analogue of Proposition \ref{proposition:edge-action-from-front-faces} holds: under \textup{(S1)} and \textup{(S2)}, the graded Leibniz rule implies that the edge action agrees with the \emph{terminal-face} component of the differential, with sign \((-1)^{q+1}\). Concretely, for \(e\in S_1\) and \(\tau \in S_q\) with \( \tau \star e \) defined, \[ \pi_{\tau \star e} \circ (- \cdot \omega_e) \circ \iota_\tau = (-1)^{q+1}D_{\tau \to \tau \star e}, \quad \pi_\zeta \circ( - \cdot \omega_e) \circ \iota_\tau = 0
\quad (\zeta \neq \tau \star e).\] The sign arises because \(\tau = d_{q+1}( \tau \star e)\) is the last face of \(\tau\star e\), whereas in the left case \(\tau=d_0(e\star\tau)\) is the first face.
\end{remark}

Write \({}^\mathsf{c}\mathcal{M}_S^{\mathrm{DG}}\) for the subcategory of \( {}\mathcal{M}_S^{\mathrm{DG}}\) consisting of cellular right modules with cellular morphisms. For
\(M\in {}^\mathsf{c}\mathcal{M}_S^{\mathrm{DG}}\), define a cellular sheaf
\(F_M^{\mathrm r}\in\mathsf{Sh}(S)\) by \(F_M^{\mathrm r}(\tau)\coloneqq M(\tau)\)
and \(F_M^{\mathrm r}(\tau\to\zeta)\coloneqq (-1)^iD_{\tau\to\zeta}\)
whenever \(d_i\zeta=\tau\), extended to general inclusions via signed
composition along maximal chains exactly as in Lemma
\ref{lemma: sheaf associated to a cellular differential graded module is functorial and well-defined}.

The assignments above define functors \( \mathcal R^{\mathrm r}\colon \mathsf{Sh}(S)\rightarrow {}^{\mathsf c}\mathcal M_S^{\mathsf{DG}}, \ \mathcal S^{\mathrm r}\colon {}^{\mathsf c}\mathcal M_S^{\mathsf{DG}} \rightarrow\mathsf{Sh}(S).\)  Mirroring Theorem \ref{theorem: equivalence between cellular sheaves and differential graded module} shows that these functors define equivalences. In fact, the left and right realizations are two presentations of the same sheaf data. For \(x \in F(\tau)\), the actions \( \rho_\alpha \cdot x \coloneqq F(\tau\to\alpha\star\tau) x \) and \( x \cdot \rho_\beta \coloneqq F( \tau \to \tau \star \beta ) x \) (zero when the concatenation does not exist) are compatible by functoriality, so every cellular sheaf determines a DG-bimodule over \(\Omega_S^\bullet\). 

More generally, the restriction maps \(\{F(\tau\leq\xi)\}_{\tau\leq\xi\text{ in }P_S}\) define a family of graded endomorphisms of \(M_F\) indexed by all face incidences, not only the initial and terminal ones that govern the left and right actions. The condition \(d_{M_F}^2=0\) becomes a system of quadratic relations among these endomorphisms. The left and right actions single out the front and terminal coface blocks, respectively, while the interior coface blocks retain the remaining incidence data. We leave a systematic study of the resulting graded endomorphism algebra to future work.

\section{Connections and curvature over \(\Omega_S^\bullet\)}\label{section: connections and curvature} In this section we study connections on finite-dimensional left \(A_S\)-modules relative to the differential graded algebra \(\Omega_S^\bullet\) and compute their curvature. We use the standard definitions of connections and curvature relative to a differential graded algebra, see e.g. \cite{BeggsMajid2020QRG, DuboisVioletteEtAl1995Curvature, Landi1997Introduction} for additional background. We adopt the same conventions as in \S\ref{section: cellular sheaves, differential modules}, namely the underlying base \(S\in\ssreg_{\leq n}\) is kept fixed and modules are finite-dimensional over \( \Bbbk \). The combinatorial structure of \( \Omega_S^\bullet \) gives a geometric reading of these notions: connections are determined by edge maps (Proposition \ref{prop:edgewise-connections}), while curvature on a \(2\)-simplex is the difference between direct and composite edge transport (Proposition \ref{proposition: curvature written out for DG modules for AW DGA in terms of transport operators}). Closely related transport and curvature formulas occur in \cite{BrauneTongGayBalmazDesbrun2024DEC} and \cite[\S IV]{DimakisMullerHoissen1994Discrete}.

\begin{excerpt}
Recall that, since $A_S$ is finite-dimensional and semisimple, every finite-dimensional left $A_S$-module $M$ is finitely generated projective as a left $A_S$-module. We write $\prescript{}{A_S}{ \operatorname{Hom} }(- , -) $ for spaces of left $ A_S $-linear maps and $\mathrm{Hom}_{A_S}(-,-)$ for right $A_S$-linear ones. We also denote \(  M^{\vee}:=\prescript{}{A_S}{\operatorname{Hom}}(M,A_S),\) seen as right $A_S$-module via $(\varphi\cdot f)(x):=\varphi(x) f$. The following lemma will be used repeatedly. 
\end{excerpt}

\begin{lemma}\label{lem:coeff}
Let \(M\) be a finite-dimensional left \(A_S\)-module and \(p\geq 0\). Then
\[ \Theta_p: M^\vee\otimes_{A_S}\Omega_S^p\otimes_{A_S}M \rightarrow \prescript{}{A_S}{\operatorname{Hom}} \bigl(M,\Omega_S^p\otimes_{A_S}M\bigr), \qquad \Theta_p(\varphi \otimes \beta \otimes m)(x) = \varphi(x)\beta\otimes m, \]
is an isomorphism of \(A_S\)-bimodules, where the bimodule structures are induced by the left and right \(A_S\)-actions on \(\Omega_S^p\). Moreover,
\[ \prescript{}{A_S}{\operatorname{Hom}} \bigl(M,\Omega_S^p\otimes_{A_S}M\bigr) \cong \bigoplus_{\sigma = [v_0, \ldots, v_p]\in S_p} \operatorname{Hom}_{\Bbbk} \bigl( M_{v_0} , M_{v_p}\bigr), \] under which a family \(U_\sigma \colon M_{v_0} \to M_{v_p}\) corresponds to \( x \to  \sum_{\sigma=[v_0, \ldots, v_p] \in S_p} \rho_\sigma\otimes U_\sigma( \delta_{v_0} x).\)
\end{lemma}

\begin{proof}
Choose a left dual basis \(\{(m_i,\lambda^i)\}_i\) of \(M\). Since \(M\) is finitely generated and projective, for every left \(A_S\)-module \( N \), the assignment
\[ M^\vee\otimes_{A_S}N \rightarrow \prescript{}{A_S}{\operatorname{Hom}}(M,N),
\qquad (\varphi\otimes n)(x) = \varphi(x)\cdot n, \]
is an isomorphism, with inverse \(F\rightarrow\sum_i\lambda^i\otimes F(m_i)\). Taking \( N = \Omega_S^p\otimes_{A_S}M\) and using the natural left action on \( \Omega_S^p \otimes_{A_S} M\) gives the formula for \(\Theta_p\). Both sides carry \(A_S\)-bimodule structures.
Since \(A_S\) is commutative, the \(A_S\)-bimodule structure of
\(\Omega_S^p \) induces an \( A_S \)-bimodule structure on
\( \Omega_S^p\otimes_{A_S}M\), given by \(f \cdot(\beta\otimes x)\cdot g = (f\beta g)\otimes x. \) These induce \(  A_S \)-bimodule structures on \( M^\vee\otimes_{A_S}\Omega_S^p\otimes_{A_S}M \) and \( \prescript{}{A_S}{\operatorname{Hom}} (M,\Omega_S^p\otimes_{A_S}M) \), given by \[ a \cdot ( \varphi \otimes \beta \otimes m) \cdot b = \varphi \otimes a \beta b \otimes m, \qquad ( a \cdot F \cdot b)(x) = a \cdot F(x) \cdot b, \] respectively. Then \(\Theta_p\) defines an isomorphism of \(A_S\)-bimodules.

For the second statement, we know that \( \Omega_S^p \otimes_{A_S} M \cong \bigoplus_{\sigma = [v_0,\ldots,v_p]\in S_p} \Bbbk \rho_\sigma \otimes_{A_S}M.\) Since \( \rho_\sigma\otimes m = \rho_\sigma\otimes\delta_{v_p}m,\) we can identify \(\Bbbk \rho_\sigma \otimes_{A_S}M \) with \(M_{v_p}\). As a left \(A_S\)-module, however, \(\Bbbk \rho_\sigma \otimes_{A_S}M\) is supported at \(v_0\) since \( \delta_v \cdot (\rho_\sigma \otimes m) =\delta_{v,v_0} \rho_\sigma \otimes m. \) Thus, if \( F\colon M \to \Bbbk \rho_\sigma \otimes_{A_S} M \) is left \( A_S \)-linear and \(x \in M_v \), then \( F(x) = \delta_v F(x),\) which vanishes unless \(v=v_0\). Thus \(F\) is uniquely determined by its restriction to \(M_{v_0}\), and, since \(\Bbbk \rho_\sigma \otimes_{A_S}M\cong M_{v_p}\), this restriction is an arbitrary \(\Bbbk\)-linear map
\(U_\sigma\colon M_{v_0}\rightarrow M_{v_p}.\)

Conversely, such a map determines a left \(A_S\)-linear map \(M\to \Bbbk \rho_\sigma \otimes_{A_S}M\) by \(x \to \rho_\sigma \otimes U_\sigma(\delta_{v_0}x).\) Hence \( \prescript{}{A_S}{\operatorname{Hom}}\bigl(M,\Bbbk \rho_\sigma \otimes_{A_S}M\bigr) \cong \operatorname{Hom}_{\Bbbk}\bigl( M_{v_0}, M_{v_p} \bigr),\) and taking the direct sum over \(\sigma\in S_p\) gives the result.
\end{proof}

\begin{definition}\label{definition: left connection}
A \emph{left connection} on a left \(A_S\)-module \(M\) is a \(\Bbbk\)-linear map
\( \nabla: M \rightarrow \Omega_S^1\otimes_{A_S} M\) satisfying the \emph{Leibniz rule} \( \nabla(fx) = df \otimes x+f \cdot \nabla (x)\) for all \(f\in A_S, \ x\in M\). We denote the set of left
connections on \( M \) by \(\operatorname{Conn}_{A_S}(M)\). The connection \(\nabla\) extends to the degree-\(1\) operator \( \widehat\nabla\colon \Omega_S^\bullet\otimes_{A_S}M \rightarrow
\Omega_S^{\bullet+1}\otimes_{A_S}M \)
defined by
\[ \widehat\nabla(\alpha\otimes x) = d\alpha\otimes x + (-1)^{|\alpha|}
\alpha\wedge\nabla(x), \qquad \alpha\in\Omega_S^\bullet,\quad x\in M. \]
\end{definition}

\begin{remark}\label{remark:affine-space}
Let \( \nabla_0 \) be a left connection on \(M\). Then \[ \prescript{}{A_S}{\operatorname{Hom}} \bigl(M,\Omega_S^1\otimes_{A_S}M\bigr) \rightarrow \operatorname{Conn}_{A_S}(M), \qquad F\to \nabla_0+F\] is a bijection. Thus, whenever it is nonempty, \( \operatorname{Conn}_{A_S} (M)\) is an affine space modeled on \( \prescript{}{A_S}{\operatorname{Hom}} \bigl( M, \Omega_S^1 \otimes_{A_S} M\bigr)\). Indeed, for \(\nabla\) and \(\nabla_0\) two left connections, \[(\nabla-\nabla_0)(fx) = f\cdot(\nabla-\nabla_0)(x), \] so \(\nabla-\nabla_0\in \prescript{}{A_S}{\operatorname{Hom}} \bigl(M,\Omega_S^1\otimes_{A_S}M\bigr)\). Conversely, if \( F\in\prescript{}{A_S}{\operatorname{Hom}} \bigl(M, \Omega_S^1 \otimes_{A_S}M\bigr)\), then \(\nabla_0+F\) satisfies the Leibniz rule and defines a connection.
\end{remark}

Finitely generated projective modules admit left connections, as the following example shows.

\begin{example}\label{example: grassmann connection} The vertex idempotents decompose each left $A_S$-module \(M\) as \( M=\bigoplus_{v\in S_0} M_v,\ M_v \coloneq \delta_v M.\) This splitting gives a canonical \emph{Grassmann connection} \[ \nabla^{\mathrm{can}}(x) \coloneq \sum_{v\in S_0} d\delta_v\otimes \delta_v x = \sum_{e\in S_1}\omega_e\otimes \delta_{d_0(e)}x. \] 

The second equality follows from the incidence formula and the identities \( \omega_e \otimes \delta_{d_1(e)}x=0\) and \(\omega_e\otimes x=\omega_e\otimes\delta_{d_0(e)}x\). Moreover,
\[ \begin{aligned} df\otimes x+f\cdot\nabla^{\mathrm{can}}(x) & = \sum_{e\in S_1}
\bigl(f(d_0(e))-f(d_1(e))+f(d_1(e))\bigr)
\omega_e\otimes\delta_{d_0(e)}x=
\nabla^{\mathrm{can}}(fx),
\end{aligned}\]
Thus \(\nabla^{\mathrm{can}}\) is a left connection. In particular,
\(\operatorname{Conn}_{A_S}(M)\neq\varnothing\).
\end{example}

\begin{remark}\label{rem:not-end-valued}
The ordered tensor product in Lemma \ref{lem:coeff} is the usual finite-projective realization of the coefficient space
\(\prescript{}{A_S}{\operatorname{Hom}}\bigl( M,\Omega_S^p\otimes_{A_S}M\bigr).\)
It should not be confused with an endomorphism-valued \(p\)-form. Classically, if \(A\) is commutative and \(\Omega^p\) is a central \(A\)-bimodule, one has
\[ E^\vee\otimes_A\Omega^p\otimes_AE \cong \Omega^p\otimes_A\bigl(E^\vee\otimes_AE\bigr) \cong \Omega^p\otimes_A \prescript{}{A}{\operatorname{End}}(E) \](cf. \cite[\S 7.2]{Landi1997Introduction}). 
In general, the permutation \( \varphi \otimes \beta \otimes m \rightarrow \beta \otimes \varphi \otimes m \) descends to the tensor products over \(A_S\) only when \(\beta\) is central. Indeed,
applying it to the two sides of the balancing relation \( ( \varphi \cdot f) \otimes \beta \otimes m = \varphi \otimes f \beta \otimes m\) gives \( \beta \otimes( \varphi \cdot f)\otimes m\) and \(f \beta \otimes \varphi\otimes m\). Since \( A_S \) is commutative, the former equals \( \beta f \otimes \varphi \otimes m\). The two agree for every \(f\in A_S\) precisely when \( \beta f = f\beta\). For \(\Omega_S^\bullet\), this condition is not satisfied.

\end{remark}

We now describe connections in terms of edgewise coordinates. For each \(e\in S_1\), define
\begin{equation}\label{eq:right-dual-base}
\varphi^e \in \mathrm{Hom}_{A_S}(\Omega_S^1, A_S), \quad \varphi^e(\omega_{e'}) = \delta_{e,e'} \delta_{d_0(e)}.
\end{equation}
One verifies that each \(\varphi^e\) is right \(A_S\)-linear and that
\(\sum_{e \in S_1} \omega_e \varphi^e = \mathrm{id}_{\Omega_S^1}\) as right
\(A_S\)-linear maps, so that \(\{(\omega_e, \varphi^e)\}_{e \in S_1}\) is a
right \(A_S\)-linear dual basis of \(\Omega_S^1\). Accordingly, for \(\phi \in \mathrm{Hom}_{A_S}(\Omega_S^1, A_S)\), we write \((\phi \otimes \mathrm{id}_M) \colon \Omega_S^1 \otimes_{A_S} M \to M\) for the composite of \(\phi \otimes \mathrm{id}_M\) with \(\mu_M\). In particular, for any \(y \in \Omega_S^1 \otimes_{A_S} M\), \[ y = \sum_{e \in S_1} \omega_e \otimes (\varphi^e \otimes \mathrm{id}_M)(y). \]

\begin{proposition}\label{prop:edgewise-connections} Let \(M\) be a finite-dimensional left \(A_S\)-module and choose \(\{( \omega_e, \varphi^e)\}_{ e \in S_1}\) as in \eqref{eq:right-dual-base}. Let \( \nabla \colon M \rightarrow \Omega_S^{1}\otimes_{A_S} M\) be a left connection and define \( T_e  \coloneq  (\varphi^e \otimes \mathrm{id}_M)\circ \nabla \colon M \rightarrow M. \) Then:
\begin{enumerate}[nosep]
\item For all \(x\in M\), \( \nabla(x) = \sum_{e\in S_1} \omega_e \otimes T_e(x) \) uniquely and \( T_e(x) \in M_{ d_0(e) } \) for every \(x\in M\).
\item For \(f \in A_S\) and \(x\in M\),
\begin{equation}\label{equation: partial leibniz for the endomorphisms T_e} 
T_e(fx)=\partial_e(f)\cdot x + f(d_1(e))T_e(x),
\quad \partial_e(f)  \coloneq  \varphi^e(df)=\bigl(f(d_0(e)) - f(d_1(e))\bigr)\delta_{d_0(e)}.
\end{equation}
\item For each $e\in S_1$, there is a unique map \( R_e\in \operatorname{Hom}_\Bbbk\bigl(M_{d_1(e)},M_{d_0(e)}\bigr)\) such that \[
T_e=\delta_{d_0(e)}-R_e\delta_{d_1(e)}.
\]  Conversely, given a family
\(\{R_e\in \operatorname{Hom}_{\Bbbk}\bigl(M_{d_1(e)},M_{d_0(e)}\bigr)\}_{e\in S_1},\)
the formula \[\nabla(x)  \coloneq  \sum_{e\in S_1}\omega_e\otimes \bigl(\delta_{d_0(e)}-R_e\delta_{d_1(e)}\bigr)(x) \]
defines a unique left connection on \(M\).
\end{enumerate}
\end{proposition}

\begin{proof}
The first statement follows by applying the formulas from \eqref{eq:right-dual-base} to \(y=\nabla(x)\). Then \(T_e(x)\in M_{d_0(e)}\) follows since \(\varphi^e(\Omega_S^1)\subseteq \delta_{d_0(e)}A_S\).  For the second assertion, the Leibniz rule gives  
\[ \begin{aligned}
T_e(fx)
&=(\varphi^e\otimes \mathrm{id}_M)(df\otimes x)+(\varphi^e\otimes \mathrm{id}_M)\bigl(f\cdot \sum_{e'\in S_1} \omega_{e'}\otimes T_{e'}(x)\bigr)\\
&=\varphi^e(df)\cdot x+\sum_{e'\in S_1} f(d_1(e'))\varphi^e(\omega_{e'})T_{e'}(x)=\partial_e(f)x + f(d_1(e))T_e(x).\\
\end{aligned} \] For the last part, apply \eqref{equation: partial leibniz for the endomorphisms T_e} with \(f=\delta_u\). Then \[ T_e(\delta_u x) = \bigl(\delta_u(d_0(e))-\delta_u(d_1(e))\bigr)\delta_{d_0(e)}x + \delta_u(d_1(e)) T_e(x). \] Hence \(T_e(\delta_u x)=0\) if \(u\notin\{d_0(e),d_1(e)\}\). For \(u=d_0(e)\), we get \( T_e(\delta_{d_0(e)}x) = \delta_{d_0(e)}x, \) so \(T_e\lvert_{M_{d_0(e)}}=\operatorname{id}_{M_{d_0(e)}}\). Since \(T_e(x)\in M_{d_0(e)}\), the restriction \( R_e \coloneq - T_e\lvert_{M_{d_1(e)}} \in \operatorname{Hom}_\Bbbk(M_{d_1(e)},M_{d_0(e)}) \)
is well-defined. Thus, for \(x=\sum_{v\in S_0}\delta_vx\), we have
\[ \begin{aligned}
T_e(x) &=T_e (\delta_{d_0(e)}x)+T_e(\delta_{d_1(e)}x)+\sum_{u\notin\{d_0(e),d_1(e)\}}T_e(\delta_ux) = \delta_{d_0(e)}x-R_e(\delta_{d_1(e)}x).
\end{aligned}\] Uniqueness of \(R_e\) is immediate by restricting to \(M_{d_1(e)}\).

Conversely, define \( T_e  \coloneq  \delta_{d_0(e)} - R_e \delta_{d_1(e)}.\) Then \(T_e(x)\in M_{d_0(e)}\) for all \(x\in M\). Also, for \( f \in A_S\) and \( x \in M\),
\[ \begin{aligned}
T_e(fx)
&=f(d_0(e)) \delta_{d_0(e)}x-f(d_1(e)) R_e\delta_{d_1(e)}x\\
&=\bigl(f(d_0(e))-f(d_1(e))\bigr)\delta_{d_0(e)}x
+f(d_1(e))\bigl(\delta_{d_0(e)}-R_e\delta_{d_1(e)}\bigr)(x)\\
&=\partial_e(f)\cdot x+f(d_1(e)) T_e(x).
\end{aligned} \]
Therefore \( \nabla(x) \coloneq \sum_{e\in S_1}\omega_e\otimes T_e(x)\)
defines a \(\Bbbk\)-linear map \(M\to \Omega_S^1\otimes_{A_S}M\), and
\[ \begin{aligned}
\nabla(fx)
&=\sum_{e\in S_1}\omega_e\otimes T_e(fx)= \sum_{e\in S_1}\omega_e\otimes \partial_e(f)x
+\sum_{e\in S_1}\omega_e\otimes f(d_1(e))T_e(x)\\
&=\Bigl(\sum_{e\in S_1}(f(d_0(e))-f(d_1(e)))\omega_e\cdot \delta_{d_0(e)}\Bigr)\otimes x
+\sum_{e\in S_1}(f\cdot \omega_e)\otimes T_e(x)\\
&=df\otimes x + f\cdot \nabla(x).
\end{aligned}
\]
So \(\nabla\) is a left connection. If \(\nabla'\) is another connection with \((\varphi^e\otimes \operatorname{id}_M)\circ \nabla'=T_e\) for all \(e\in S_1\), then
\[ (\varphi^e\otimes \operatorname{id}_M)\bigl(\nabla'(x)-\nabla(x)\bigr)=0
\quad\forall x\in M,\ e\in S_1, \]
hence \(\nabla'(x)-\nabla(x)=0\) by the right-dual basis property. Thus \(\nabla\) is unique.
\end{proof}

\begin{corollary}\label{corollary:connection-space}
The assignment \(\nabla\to(R_e)_{e\in S_1}\) of Proposition \ref{prop:edgewise-connections} defines an affine isomorphism \(\operatorname{Conn}_{A_S}(M) \cong \prod_{e\in S_1} \operatorname{Hom}_{\Bbbk} \bigl(M_{d_1(e)},M_{d_0(e)}\bigr). \) Under this identification, the canonical Grassmann connection corresponds to the zero family. More precisely, under the simplexwise identification of Lemma \ref{lem:coeff}, the relative map \(\nabla-\nabla^{\mathrm{can}}\) corresponds to the family \((-R_e)_{e\in S_1}\). \end{corollary}

\begin{proof}
The affine bijection follows directly from Proposition
\ref{prop:edgewise-connections}. For the final statement,
\[ \begin{aligned}
\bigl(\nabla-\nabla^{\mathrm{can}}\bigr)(x) & =
\sum_{e\in S_1} \omega_e\otimes \left( \delta_{d_0(e)} - R_e\delta_{d_1(e)} - \delta_{d_0(e)}
\right)(x) \\ & = -\sum_{e\in S_1} \omega_e\otimes R_e\bigl(\delta_{d_1(e)}x\bigr).
\end{aligned} \] Comparison with the formula of Lemma \ref{lem:coeff} gives \(U_e=-R_e\).
\end{proof}

\begin{definition}\label{definition:curvature} Let \(\nabla \in \operatorname{Conn}_{A_S}(M)\) and \(\widehat \nabla \) its extension from Definition \ref{definition: left connection}. The \emph{curvature} of \(\nabla\) is the map \( K_\nabla \coloneq \widehat\nabla^{2}\big|_M \colon M\rightarrow \Omega_S^2 \otimes_{A_S} M.\) \end{definition}

The cancellation of the two Leibniz terms shows that \(K_\nabla\) is left \(A_S\)-linear (cf. \cite[Lemma 3.19]{BeggsMajid2020QRG}). Hence
\( K_\nabla \in \prescript{}{A_S}{ \operatorname{Hom}} 
\bigl( M, \Omega_S^2 \otimes_{A_S} M\bigr)\) and through \( \Theta_2^{-1}\), it can equivalently be represented in \( M^\vee \otimes_{A_S} \Omega_S^2 \otimes_{A_S} M\). We write \((K_\nabla)_\sigma\) for the component of \(K_\nabla\)
indexed by \(\sigma\in S_2\). Applying \(\widehat \nabla\) to the edgewise expansion from Proposition \ref{prop:edgewise-connections} gives
\begin{equation}\label{equation: curvature expansion}
K_\nabla(x) = \sum_{e\in S_1}d\omega_e\otimes T_e(x) - \sum_{e,e'\in S_1}
\omega_e\wedge\omega_{e'} \otimes T_{e'}\bigl(T_e(x)\bigr).
\end{equation}

\begin{proposition}\label{proposition: curvature written out for DG modules for AW DGA in terms of transport operators}
Let \(\nabla\) be a left connection and let \( \sigma = [v_0,v_1,v_2] \in S_2\). Write \( e_{01}^{ \sigma} \coloneq d_2\sigma,\ e_{02}^{\sigma} \coloneq d_1\sigma,\ e_{12}^{\sigma} \coloneq d_0\sigma, \) and, for \(0 \leq i < j \leq 2\), set \( R_{ij}^{\sigma} \coloneq R_{e_{ij}^{ \sigma} } \colon M_{v_i} \rightarrow M_{v_j}. \) Then \( ( K_\nabla )_\sigma = R_{02}^{\sigma} - R_{12}^{\sigma} R_{01}^{\sigma} \colon M_{v_0} \rightarrow M_{v_2}.\)
Equivalently,\[ K_\nabla(x) = \sum_{\sigma=[v_0,v_1,v_2]\in S_2}
\rho_\sigma\otimes \bigl( R_{02}^{\sigma } - R_{12}^{\sigma}R_{01}^{\sigma}
\bigr)(\delta_{v_0}x).\]
\end{proposition}

\begin{proof}
For the proof, write \(T_{ij}^{\sigma}\coloneq T_{e_{ij}^{\sigma}}\), \(\omega_{ij}^{\sigma}:=\omega_{e_{ij}^{\sigma}}\). The \(\rho_\sigma\)-contributions to \eqref{equation: curvature expansion} are \( +d\omega_{01}^\sigma,\) \(-d\omega_{02}^\sigma, + d\omega_{12}^\sigma\), together with \( \omega_{01}^\sigma \wedge \omega_{12}^\sigma = \rho_\sigma \). Hence the \(\sigma\)-contribution to \( K_\nabla \) is \[ \rho_\sigma \otimes Q_\sigma, \qquad Q_\sigma=T^\sigma_{01}-T^\sigma_{02}+T^\sigma_{12}-T^\sigma_{12}T_{01}^\sigma.\] Since \(\rho_\sigma = \rho_\sigma \delta_{v_2},\) we have \(\rho_\sigma \otimes Q_\sigma(x) = \rho_\sigma \otimes \delta_{v_2} Q_\sigma(x). \) Moreover, under the  identification of Lemma \ref{lem:coeff}, the coefficient of \(\rho_\sigma\) is obtained by restricting the input to \(M_{v_0}\). Hence \( (K_\nabla)_\sigma = \delta_{v_2} Q_\sigma \delta_{v_0}.\) By Proposition \ref{prop:edgewise-connections}, \( T^\sigma_{ij} = \delta_{v_j}-R^\sigma_{ij} \delta_{v_i},\) so only the terms \( \delta_{v_2} T^\sigma_{02} \delta_{v_0} = -R^\sigma_{02} \) and \( \delta_{v_2} T^\sigma_{12} T^\sigma_{01} \delta_{v_0} = (-R^\sigma_{12})(-R^\sigma_{01}) = R^\sigma_{12}R^\sigma_{01}\) remain. Hence \( (K_\nabla)_\sigma = R^\sigma_{02}-R^\sigma_{12}R^\sigma_{01}\) and the result follows after summation.
\end{proof}

Hence, the curvature component on a triangle is precisely the difference between direct and composite transport. Consequently, \(K_\nabla\) vanishes on \(\sigma=[v_0,v_1,v_2]\) if and only if \( R^\sigma_{02}=R^\sigma_{12}R^\sigma_{01}.\)

\begin{excerpt}\label{excerpt: right connections and bimodule structure} The above has a natural right-handed counterpart. For a right \(A_S\)-module \(N\), write \(N_v:=N\delta_v\). Right connections \(\widetilde \nabla \colon N \rightarrow N \otimes_{A_S} \Omega_S^1\) are in bijection with families \[\left\{ R_e^{\mathrm r}\in \operatorname{Hom}_{\Bbbk} \bigl( N_{d_0(e)}, N_{d_1(e)}\bigr) \right\}_{e\in S_1}.\] Explicitly, \[  \widetilde \nabla(x) = \sum_{e\in S_1} \left( R_e^{\mathrm r}(x\delta_{d_0(e)}) - x\delta_{d_1(e)} \right)\otimes \omega_e.\] For \(\sigma=[v_0,v_1,v_2]\in S_2\), use the notation \(e_{ij}^{\sigma}\) of Proposition \ref{proposition: curvature written out for DG modules for AW DGA in terms of transport operators}
and set \(R_{ij}^{\mathrm r,\sigma} \coloneq R_{e_{ij}^{\sigma}}^{\mathrm r}. \)
Then \( (\widetilde K_{\widetilde\nabla})_\sigma = R_{01}^{\mathrm r,\sigma}R_{12}^{\mathrm r,\sigma} - R_{02}^{\mathrm r,\sigma} \colon N_{v_2}\rightarrow N_{v_0}.\) Notice that the composition order is reversed relative to the left-handed component, yet both expressions measure the defect between direct and composite transport.

The left and right versions are independent when their underlying modules are unrelated. To compare them on the same underlying vector space, use the commutativity of \(A_S\) to equip a left \(A_S\)-module \(M\) with right action \(x \cdot f\coloneq f\cdot x.\) Let \( \nabla \) be the left connection corresponding under Proposition \ref{prop:edgewise-connections} to the edge transports \( R_e \colon M_{d_1(e)}\rightarrow M_{d_0(e)}. \) Define \[ \xi_\nabla \colon M\otimes_{A_S} \Omega_S^1 \rightarrow \Omega_S^1 \otimes_{A_S} M, \qquad \xi_\nabla (x \otimes \omega_e) = \omega_e \otimes R_e \bigl( \delta_{d_1(e)} x \bigr).\] Since \( f\omega_e=f(d_1(e)) \omega_e, \ \omega_e f = f(d_0(e))\omega_e \), this is a morphism of \(A_S\)-bimodules. Moreover, Proposition \ref{prop:edgewise-connections} gives \[ \nabla( x \cdot f) - \nabla(x) \cdot f = \sum_{e\in S_1} \bigl( f(d_0(e)) - f(d_1(e) )\bigr) \omega_e \otimes R_e\bigl(\delta_{d_1(e)}x\bigr)  = \xi_\nabla(x\otimes df),\] so \((M, \nabla, \xi_\nabla)\) is a left bimodule connection in the sense of \cite[\S 3.4]{BeggsMajid2020QRG}.

Suppose that every edge transport \(R_e \colon M_{d_1(e)} \rightarrow M_{d_0(e)}\)
is an isomorphism. Then also \( \xi_\nabla\) is an isomorphism with \( \xi_\nabla^{-1}( \omega_e \otimes y) = R_e^{-1} \bigl( \delta_{ d_0(e) } y \bigr) \otimes \omega_e.\) Hence, by \cite[Lemma 3.70]{BeggsMajid2020QRG}, \( \widetilde \nabla\coloneq\xi_\nabla^{-1}\circ\nabla \) defines a right bimodule connection \((M,\widetilde\nabla,\xi_\nabla^{-1})\). Its right transports are \(R_e^{\mathrm r}=R_e^{-1}\), and explicitly \[ \widetilde \nabla(x) = \sum_{e\in S_1} \left( R_e^{-1}(x \delta_{d_0(e)}) - x \delta_{d_1(e)} \right) \otimes\omega_e. \] For an edge \(e\), the map \(R_e^{-1}\) denotes backward transport through that same edge. If \(e'\in S_1\) with \(d_1(e')=d_0(e), \ d_0(e')=d_1(e), \) then \(R_{e'}\) and \( R_e^{-1}\) have the same source and target, but no relation between them is imposed.

For \( \sigma \in S_2 \), write \( C_\sigma^\ell = R_{02}^{\sigma} - R_{12}^{\sigma} R_{01}^{\sigma}, \ C_\sigma^{\mathrm r} = ( R_{01}^{\sigma})^{-1} (R_{12}^{\sigma})^{-1} - (R_{02}^{\sigma})^{-1}.\) They are related by \[ C_\sigma^{\mathrm r} = (R_{01}^{\sigma})^{-1} ( R_{12}^{\sigma})^{-1} C_\sigma^\ell (R_{02}^{\sigma})^{-1},\] so \(C^\ell_\sigma\) and \(C^{\mathrm r}_\sigma\) vanish simultaneously. Thus the induced right connection is flat if and only if the original left connection is flat. Both measure the holonomy \( H_\sigma \coloneq ( R_{02}^{\sigma})^{ -1 } R_{ 12 }^{ \sigma} R_{01}^{ \sigma } \in \operatorname{Aut}_{\Bbbk} (M_{v_0})\), since \(( R_{02} ^{\sigma}) ^{-1}C_\sigma^ \ell = I - H_\sigma\) and \(C_\sigma^{\mathrm r} R^\sigma_{02} = H_\sigma^{-1}-I\).
We adopt the left-handed convention throughout, cf. \cite{BrauneTongGayBalmazDesbrun2024DEC, fioresi2026sheaves} for the right-handed one.
\end{excerpt}

\section{Connection sheaves and connections}\label{section: connection sheaves and connections} Sections \S\S \ref{section: cellular sheaves, differential modules}--\ref{section: connections and curvature} consider cellular DG-modules that admit arbitrary linear restriction maps. A substantial part of the applied sheaf-based learning literature instead restricts to the case where these maps are invertible, and often further constrains them to be orthogonal, motivated by the analogy with parallel transport on Riemannian manifolds, see for instance \cite{bamberger2025bundle, Barbero2022SheafConnectionLaplacians, Battiloro2022TangentBundleFilters, bodnar2022NeuralSheafDiffusion, SingerWu2012VDM}. Such sheaves appear in \cite{HansenGhrist2019Spectral} under the name \emph{discrete vector bundles}. Following \cite{Ayzenberg2025SheafTheory}, we call them \emph{connection sheaves}. We devote this section to this class and to its interaction with the curvature theory of \S\ref{section: connections and curvature}. The main statement is the extension criterion in Proposition \ref{proposition: edge transport extension criterion}, telling us that a family of vertex spaces and edge isomorphisms extends to a connection sheaf on \( S \) precisely when its curvature vanishes on every \(2\)-simplex, and the extension is then unique up to cellular gauge.

\begin{definition}\label{definition: connection sheaf}
A \emph{connection sheaf} on \(S\in\ssreg\) is a cellular sheaf \(F\colon P_S\rightarrow \mathrm{Vect}_\Bbbk\) such that \(F( \sigma \leq \tau) \colon F( \sigma) \rightarrow F( \tau)\) is an isomorphism for all \( \sigma \leq \tau\). A \emph{rank-\(d\) connection sheaf} is a connection sheaf whose stalks all have dimension \(d\).

If \(V\) is a fixed \(d\)-dimensional vector space and \(G\subseteq \mathrm{GL}(V)\), a \emph{framed \(G\)-connection sheaf} is a rank-\(d\) connection sheaf together with identifications \(F (\sigma) \cong V\) such that all structure maps are represented by elements of \(G\). We denote by \(\mathsf{CSh}(S)\subseteq \mathsf{Sh}(S)\) the full subcategory of connection sheaves.
\end{definition}

\begin{excerpt} A \emph{path in \(P_S\)} is a finite sequence \( \gamma=(\sigma_0,\sigma_1,\ldots,\sigma_n)\) of simplices such that every two consecutive simplices are comparable. For a connection sheaf $F$ on \(S\in\ssreg\) and comparable \(\alpha,\beta\in P_S\), set  \[ F(\alpha\lessgtr\beta) \coloneq
\begin{cases}
F(\alpha \leq \beta),&\alpha \leq \beta, \\[2mm]
F(\beta \leq \alpha)^{-1},&\alpha \geq\beta.
\end{cases} \]
The \emph{parallel transport} along \(\gamma\) is \[ T_\gamma \coloneq F(\sigma_{ n-1}\lessgtr\sigma_n) \circ \ldots\circ F(\sigma_0\lessgtr\sigma_1) \colon F(\sigma_0)\rightarrow F(\sigma_n).
\] Functoriality and invertibility imply that \(T_\gamma\) is unchanged under the elementary edge-path moves in the order complex \(N(P_S)\). Hence parallel transport depends only on the homotopy class of \( \gamma \) relative to its endpoints. More details can be found in \cite[\S A.4]{Ayzenberg2025SheafTheory}. For a base simplex \(\sigma\), parallel transport induces the \emph{monodromy representation} \[ \operatorname{Mon}_{F,\sigma}\colon \pi_1 \bigl( |N(P_S)|, \sigma\bigr) \rightarrow \operatorname{Aut}_{\Bbbk}\bigl(F(\sigma)\bigr).\] We say that \(F\) has \emph{trivial monodromy} if this representation is trivial.
\end{excerpt}

\begin{excerpt}\label{excerpt: transport operators for connection sheaves}
Let \(F\in\mathsf{CSh}(S)\) and \(\prescript{\mathsf{c}}{S}{\mathcal{M}}^{\mathsf{DG}} \) its associated cellular DG-module as in \S \ref{section: cellular sheaves, differential modules}. For an edge \( e \in S_1\), define \( R_e \coloneq F\bigl ( d_0(e) \leq e \bigr)^{-1} \circ F\bigl( d_1 ( e )\leq e\bigr) \colon F\bigl( d_1(e) \bigr)\rightarrow F\bigl( d_0(e) \bigr).\) By Corollary \ref{corollary:connection-space}, these edge transports determine the left connection \( \nabla_F \colon M_F^0 \rightarrow \Omega_S^1 \otimes_{A_S} M_F^0\) given by \[ \nabla_F(x) = \sum_{e\in S_1} \omega_e\otimes \left( \delta_{ d_0(e)} x - R_e \delta_{d_1(e)} x \right). \] For \( y \in M_F^0 \), the action map introduced in \S \ref{section: cellular sheaves, differential modules} satisfies \( \mu_{M_F} ( \omega_e \otimes y) = F\bigl( d_0(e) \leq e \bigr) \bigl( \delta_{ d_0 (e) } y \bigr) \in F(e).\)
Hence, \[ \mu_{M_F}\nabla_F(x) = \sum_{e\in S_1} \left( F \bigl( d_0 (e) \leq e \bigr) \bigl ( \delta_{ d_0 (e)} x \bigr) - F \bigl( d_1 (e) \leq e \bigr) \bigl( \delta_{ d_1 (e) } x \bigr) \right) = d_{M_F}|_{M_F^0}(x).\] For \( \sigma = [v_0, v_1, v_2] \in S_2 \), let \( e_{01}^\sigma = d_2\sigma,\ e_{02}^\sigma = d_1\sigma,\ e_{12}^\sigma = d_0\sigma,\) and write \( R_{ij}^\sigma \coloneq R_{e_{ij}^\sigma}\). Proposition \ref{proposition: curvature written out for DG modules for AW DGA in terms of transport operators} gives \( (K_{\nabla_F})_\sigma = R_{02}^\sigma - R_{12}^\sigma R_{01}^\sigma \colon F(v_0)\rightarrow F(v_2).\) 
\end{excerpt}

\begin{remark}\label{remark: DG module reinterpretation of connection sheaves}
Under the equivalence of Theorem \ref{theorem: equivalence between cellular sheaves and differential graded module}, connection sheaves correspond to cellular DG-modules for which every block \( D_{\tau\to\zeta}\colon M(\tau) \rightarrow M(\zeta)\) is an isomorphism. For an edge \(e\in S_1\), the degree-zero blocks recover its transport by \(R_e = -D_{d_0(e)\to e}^{-1} \circ D_{d_1(e)\to e}.\) Equivalently, the associated connection on \(M^0\) is characterized by
\(\mu_M\circ\nabla=d_M|_{M^0}.\) 
\end{remark}

\begin{excerpt}\label{excerpt: face flatness versus monodromy freeness}
Two conditions should be distinguished. Edge transport data are \emph{face-flat} if \( R_{02}^\sigma = R_{12}^\sigma R_{01}^\sigma\)  for every \(2\)-simplex \(\sigma=[v_0,v_1,v_2]\), that is, the edge data satisfy the 2-simplex cocycle condition. A connection sheaf has \emph{trivial monodromy} if its monodromy representation is trivial.

For a connection sheaf defined on the full face poset \(P_S\), face-flatness holds automatically by functoriality. Trivial monodromy additionally requires identity transport around all closed paths. It is therefore a stronger global condition. For example, on a cycle with no \(2\)-simplices, face-flatness is satisfied whereas the monodromy around the cycle may be nontrivial.

The terminology can differ in the literature. Connection sheaves with trivial monodromy are called flat in \cite{Ayzenberg2025SheafTheory}. We reserve \emph{flatness} for the local condition \(K_\nabla=0\), and use \emph{trivial monodromy} for the global condition. Conversely, for vertex and edge data not yet extended to higher-dimensional cells, face-flatness is precisely the cocycle condition governing the existence of such an extension, as we now show.
\end{excerpt}

\begin{definition}\label{definition: gauges for edge transport}
Let \(F,F'\in \mathsf{CSh}(S)\) with \(S\in\ssreg\). A \emph{gauge} from \(F\) to \(F'\) is a natural isomorphism \(h\colon F\rightarrow F'.\) Suppose that \(F\) and \(F'\) satisfy \(F(v)=F'(v)=E_v\) for every \(v\in S_0\). A gauge \( h \colon F \Rightarrow F' \) is called a \emph{cellular gauge} if \( h_v = \operatorname{id}_{E_v} \ \text{for every \(v\in S_0\)}.\)
\end{definition}

\begin{proposition}\label{proposition: edge transport extension criterion}
Let \( S \in \ssreg_{\le n}\). Suppose that for every vertex
\( v \in S_0\) we are given a finite-dimensional vector space \(E_v\), and for every
\( e\in S_1 \), an isomorphism \(R_e \colon E_{d_1(e)}\rightarrow E_{d_0(e)}.\) For every \(\sigma = [v_0, v_1, v_2] \in S_2\), use the notation \(R_{ij}^{\sigma}=R_{e_{ij}^{\sigma}}\) introduced in
Proposition \ref{proposition: curvature written out for DG modules for AW DGA in terms of transport operators}. The following are equivalent:
\begin{enumerate}[label=\textup{(\roman*)},nosep]
\item For every \(\sigma=[v_0,v_1,v_2]\in S_2\), \(R_{02}^\sigma = R_{12}^\sigma R_{01}^\sigma.\)
\item The left connection on \( M^0 := \bigoplus_{v\in S_0} E_v\)  with edge transports \(\{R_e\}_{e\in S_1}\) has curvature zero on every \(2\)-simplex.  
\item The vertex spaces and edge transports extend to a connection sheaf
\( F\colon P_S \rightarrow \mathrm{Vect}_\Bbbk \) such that \(F(v)=E_v\) for every \(v\in S_0\), and, for every \(e\in S_1\), \( R_e = F(d_0(e)\leq e)^{-1} \circ F (d_1(e)\leq e ). \)
\end{enumerate}
Moreover, such an extension is unique up to cellular gauge.
\end{proposition}
\begin{proof}
The equivalence between \textup{(i)} and \textup{(ii)} is precisely Proposition \ref{proposition: curvature written out for DG modules for AW DGA in terms of transport operators}.

We now prove that \textup{(i)} implies \textup{(iii)}. For \(\tau = [ w_0, \ldots, w_p ] \in S_p\), define \( F^{\mathrm{can}}( \tau ) := E_{w_p}. \) For \(0\leq a<b\leq p\), let \(e_{ab}^{\tau}\) denote the edge-face
of \(\tau\) with vertices \(w_a,w_b\), and set
\(R_{ab}^{\tau}:=R_{e_{ab}^{\tau}}\). If
\(\sigma\leq\tau\) and \(t(\sigma)=w_a\), define
\[
F^{\mathrm{can}}(\sigma\leq\tau):=
\begin{cases}
R_{ap}^{\tau},&a<p,\\[1mm]
\operatorname{id}_{E_{w_p}},&a=p.
\end{cases}
\]

To verify functoriality, let \(\sigma\leq \eta \leq \tau \), and write
\(t(\sigma)=w_a\) and \(t(\eta)=w_b\). Then \( a \leq b\leq p\). If \(a=b\) or \(b=p\), the required identity is immediate. Otherwise, \([w_a, w_b, w_p] \) is a \(2\)-face of \(\tau\), and condition \textup{(i)}, applied to this \(2\)-face, gives \(R_{ap}^{\tau}=R_{bp}^{\tau}R_{ab}^{\tau}.\) 
Hence, \[F^{\mathrm{can}}(\sigma\leq\tau) = F^{\mathrm{can}}(\eta\leq\tau)
\circ F^{\mathrm{can}}(\sigma\leq\eta).\] So \(F^{\mathrm{can}}: P_S \rightarrow \mathrm{Vect}_\Bbbk\) is a functor and all its restriction maps are isomorphisms. For \(e=[u,v]\in S_1\), then \( F^{\mathrm{can}}(e) = E_v, \ F^{\mathrm{can}}(v\leq e) = \operatorname{id}_{E_v}, \ F^{\mathrm{can}}(u\leq e)=R_e.\) Consequently, the induced edge transport is
\[ F^{\mathrm{can}}(v\leq e)^{-1} \circ F^{\mathrm{can}}(u\leq e)=R_e. \]

Conversely, suppose that \(F\) is an extension as in \textup{(iii)}. Let \( \sigma=[w_0,w_1,w_2] \in S_2\). For \( 0\leq a < b\leq2\), functoriality along
\(w_a\leq e_{ab}^{\sigma}\leq\sigma\) gives \( R_{ab}^{\sigma} = F(w_b\leq\sigma)^{-1} F(w_a\leq\sigma). \)
Hence
\[ R_{12}^{\sigma}R_{01}^{\sigma} = F(w_2\leq\sigma)^{-1}F(w_1\leq\sigma) F(w_1\leq\sigma)^{-1}F(w_0\leq\sigma) = F(w_2\leq\sigma)^{-1}F(w_0\leq\sigma)
=R_{02}^{\sigma}. \]

It remains to prove uniqueness up to cellular gauge. Let $F$ be any connection sheaf extending the given vertex spaces and edge transports, and let $F^{\mathrm{can}}$ be the extension constructed above. For $\tau\in P_S$, set $h_\tau:=F(t(\tau)\leq\tau)^{-1}\colon F(\tau)\to E_{t(\tau)}=F^{\mathrm{can}}(\tau)$, with $h_v=\mathrm{id}_{E_v}$ on vertices. Let \(\sigma\leq\tau\), and put \(u=t(\sigma)\) and \(v=t( \tau)\). If \(u=v\), the same identity is immediate. If \(u \neq v\), let \( e \) be the edge-face of \(\tau\) joining \(u\) to \(v\). By assumption, \( F(v \leq \tau)^{-1} F(u \leq \tau) = R_{e} = F^{\mathrm{can}}(v \leq \tau)^{-1} F^{\mathrm{can}}(u\leq\tau). \)
Hence, \(h_\tau \circ F(u\leq\tau)=F^{\mathrm{can}}(u\leq\tau).\) 
By functoriality, \[ h_\tau\circ F(\sigma \leq\tau)\circ F(u \leq \sigma) = F^{\mathrm{can}}(\sigma \leq \tau) \circ h_\sigma \circ F(u \leq \sigma), \] so \(h\colon F\rightarrow F^{\mathrm{can}}\) is a cellular gauge. \end{proof}

\begin{remark} Proposition \ref{proposition: edge transport extension criterion} has a natural interpretation as a gluing statement on the Alexandrov space \(X_{P_S}\). Let \( \underline{E_v} \) denote the constant sheaf on \( U_v \) with value \( E_v \). For every \(v \in S_0\), let \(U_v = \{ \tau \in P_S \mid v \leq \tau \}.\) For \( v \neq w\), the intersection decomposes as \[ U_v \cap U_w = \bigsqcup_{ \substack{ e \in S_1 \\
\{d_1(e), d_0(e)\} = \{ v, w\}}} U_e. \] On each component \( U_e\), define the transition isomorphism \( \varphi_{vw}^e \colon \left. \underline{E_v} \right|_{ U_e} \rightarrow \left. \underline{ E_w } \right|_{ U_e } \) by \[ \varphi^e_{vw} = \begin{cases} \underline{R_e}, & (d_1(e), d_0(e))=(v,w), \\
\underline{R_e^{-1}}, & (d_1 (e), d_0(e))=(w,v). \end{cases}\] 
This determines \( \varphi_{vw} \) on \( U_v \cap U_w \) such that \( \varphi_{wv} = \varphi_{vw}^{-1}\). If both \(e = [v,w]\) and \( e' = [w,v] \) occur, then \( \varphi^e_{vw} = \underline{R_e}, \ \varphi^{e'}_{vw} = \underline{R_{e'}^{-1}}. \) No relation between \( R_e \) and \( R_{e' } \) is imposed, since they belong to different components of the overlap.

Similarly, the opens \(U_\sigma\), with \(\sigma\) ranging over the
\(2\)-simplices having vertex set \(\{v_0,v_1,v_2\}\), are precisely
the components of \(U_{v_0}\cap U_{v_1}\cap U_{v_2}\). For every such \(2\)-simplex \(\sigma=[v_0,v_1,v_2]\) and \(0 \leq i < j \leq2\), let
\( e_{ij}^{\sigma} = [v_i,v_j ]\) and set
\( R_{ ij }^{ \sigma} := R_{e_{ij}^{ \sigma}}\). Since \( U_\sigma \subseteq U_{e_{ij}^{\sigma}},\) the restriction of \(\varphi_{v_iv_j}\) to \(U_\sigma\) is the morphism between constant sheaves induced by \( R_{ij}^{\sigma} \colon E_{v_i} \rightarrow E_{v_j}. \)
Equivalently, for every \(\tau\in U_\sigma\), \( (\varphi_{v_iv_j})_\tau = R_{ij}^{\sigma}. \) Hence, the Čech cocycle condition on \(U_\sigma\), \(\varphi_{v_0v_2} = \varphi_{v_1v_2}\circ\varphi_{v_0v_1},\)
is equivalent to \(R_{02}^{\sigma} = R_{12}^{\sigma}R_{01}^{\sigma}.\) In other words, face-flatness is exactly the descent condition for the constant sheaves \(\underline{E_v}\). Standard gluing of sheaves (see e.g. \cite[\href{https://stacks.math.columbia.edu/tag/00AK}{Tag 00AK}]{sp}) therefore gives a sheaf \(\widetilde F\) on \(X_{P_S}\), whose associated cellular sheaf is an extension of the prescribed data and hence is cellularly gauge-equivalent to \(F^{\mathrm{can}}\).\end{remark}

The description obtained in this way can be seen as the finite-space analogue of the correspondence between local systems and flat vector bundles on smooth manifolds. Since both sheaves and vector bundles are assembled from local data by descent, this analogy suggests that the discrete theory developed above can also provide finite models for differential-geometric objects. To formulate this more precisely, we recall the language of finite ringed spaces.

A \emph{finite ringed space} is a pair \((X,\mathcal O)\), where \(X\) is a finite topological space and \(\mathcal O\) is a sheaf of commutative unital rings. For \( X = X_{P_S} \) with \( S\in\ssreg_{\le n},\) such a sheaf is equivalently a functor \( \mathcal O\colon P_S \rightarrow\mathsf{CRing}\) (see \cite{fioresi2026gluing} or \cite{vakil}). An \(\mathcal O\)-module \(\mathcal M\) is quasi-coherent if and only if the map \[ \mathcal M_\sigma \otimes_{\mathcal O_\sigma} \mathcal O_\tau  \rightarrow \mathcal M_\tau \] is an isomorphism for every \(\sigma\leq\tau\). A proof of this statement can be found in \cite[\S 3]{Sancho2018HomotopyFiniteRingedSpaces}. We denote by \(\underline{\Bbbk}_{X_{P_S}}\) the constant structure sheaf on \(X_{P_S}\), whose stalk at every \( \sigma\in P_S \) is \(\Bbbk\) and whose structure maps are the identity.

\begin{proposition}\label{proposition: connection sheaf iff local system iff quasi-coherent} Let \(S \in \ssreg_{\leq n}\), \(F \colon P_S\to \mathrm{Vect}_\Bbbk\) a cellular sheaf and \(\widetilde F\) its induced sheaf on \(X_{P_S}\). The following are equivalent:
\begin{enumerate}[label=\textup{(\roman*)},nosep]
\item \(F\) is a connection sheaf.
\item \(\widetilde F\) is locally constant on \(X_{P_S}\).
\item \(\widetilde F\) is quasi-coherent as a \( \underline{\Bbbk}_{X_ {P_S}}\)-module.
\end{enumerate}
\end{proposition}
\begin{proof}
To see the equivalence between \textup{(i)} and \textup{(ii)}, fix \(\sigma\in P_S\). If all structure maps are isomorphisms, then for every \(\tau\in U_\sigma\), 
\(F(\sigma)\to F(\tau) \) is an isomorphism, so  \(\widetilde F|_{U_\sigma}\) is constant. Since the \(U_\sigma\)'s form a basis, \(\widetilde F\) is locally constant.

Conversely, suppose \(\widetilde F\) is locally constant. Since \(U_\sigma\) is the minimal open neighbourhood of \(\sigma\), after shrinking a trivializing neighbourhood of \( \sigma\) we may assume \(\widetilde F|_{U_\sigma}\) is constant. If \(\sigma\leq\tau\), then \(\tau\in U_\sigma\), and \(F(\sigma\leq\tau)\) is a restriction map of a constant
sheaf on \(U_\sigma\). Since \(U_\sigma\) has a least element (alternatively notice that \(U_\sigma\) is an irreducible open set for all $\sigma$), this map is an
isomorphism. Hence \(F\) is a connection sheaf.

The equivalence between \textup{(i)} and \textup{(iii)} follows by a direct application of the quasi-coherence criterion of \cite[Theorem 3.6]{Sancho2018HomotopyFiniteRingedSpaces} to the constant structure sheaf.
\end{proof}

The preceding interpretation is in terms of the finite Alexandrov space \(X_{P_S}\). Alternatively, one can think of \(S\) as obtained from a cover of an ambient space.

Let \(\mathcal U=\{U_i\}_{i\in I}\) be a finite open cover of a space \(M\), with \(I\) totally ordered. Its \emph{ordered nondegenerate Čech nerve} is the semisimplicial set \[ N^{\mathrm{ord}}(\mathcal U)_p = \bigl\{ [i_0,\ldots,i_p] \bigm| i_0<\ldots<i_p,\  U_{i_0}\cap\ldots\cap U_{i_p}\neq\varnothing \bigr\}, \] with face maps given by deleting indices. The cover is called \emph{good} if every nonempty finite intersection is contractible. Because only increasing index sequences occur, this semisimplicial set contains at most one simplex with any prescribed set of vertices. A rank-\(r\) \(\Bbbk\)-\emph{local system} on \(M\) is a locally constant sheaf of \(\Bbbk\)-vector spaces locally isomorphic to \(\underline{\Bbbk}_M^r\).  We have the following proposition.

\begin{proposition}\label{theorem: final correspondence}
Let \(M\) be a locally path-connected space and let \( \mathcal{U} = \{U_i\}_{i\in I}\) be a finite ordered good cover of \(M\). Set \(S:=N^{\mathrm{ord}}(\mathcal U).\) Then the category of rank-\(r\) connection sheaves on \(S\) is equivalent to the category of rank-\(r\)
\(\Bbbk\)-local systems on \(M\).

If moreover \(M\) is a smooth manifold and
\(\Bbbk=\mathbb R\) or \(\mathbb C\), this gives an equivalence with
the category of rank-\(r\) smooth \(\Bbbk\)-vector bundles equipped
with flat connections.
\end{proposition}

\begin{proof}
Since \(M\) is locally path-connected and \(\mathcal U\) is a good cover, every rank-\(r\) local system restricts to a constant sheaf on each nonempty finite intersection of members of \(\mathcal U\). Hence, descent data for rank-\(r\) local systems relative to \(\mathcal U\) consist of rank-\( r \) vector spaces \( E_i\), together with isomorphisms \( R_{ij} \colon E_i \rightarrow E_j \)
for \( i < j \) whenever \( U_i \cap U_j \neq \varnothing\). The transition in the reverse direction is \( R_{ij}^{-1} \). On every nonempty triple intersection \(U_i \cap U_j \cap U_k\), with \(i < j < k\), the Čech cocycle condition is \[ R_{ik} = R_{jk} R_{ij}. \]
Morphisms between two such descent data are families of linear maps \( f_i \colon E_i \to E_i'\) satisfying \( f_j R_{ij}=R_{ij}'f_i.\) By Proposition \ref{proposition: edge transport extension criterion}, these are
precisely the objects and morphisms defining rank-\(r\) connection sheaves on \(S=N^{\mathrm{ord}}(\mathcal U)\). Sheaf descent on \(\mathcal U\) therefore gives the claimed equivalence with
rank-\(r\) \(\Bbbk\)-local systems on \(M\).

For the smooth case, the standard local-system/flat-bundle correspondence \cite[\S 9.2]{Voisin_I} identifies a local system $\cL$ with \(\mathcal L\otimes_{\underline{\Bbbk}_M} C^\infty_{M,\Bbbk}\) endowed with its canonical flat connection, and conversely a flat vector bundle with its sheaf of flat sections.
\end{proof}

\begin{remark}
The preceding correspondence also tells us that the local systems involved only depend on \(2\)-truncation \(N^{\mathrm{ord}}(\mathcal U)_{\leq2}\): this is not surprising considering the cocycle conditions underlying gluing of sheaves. Under the equivalence, gauges of connection sheaves correspond to isomorphisms of local systems and, in the smooth case, to bundle isomorphisms preserving the flat connection. This provides us with a simplified combinatorial setting to study these important objects.
\end{remark}

The correspondence of Proposition \ref{theorem: final correspondence} is formulated relative to a chosen finite ordered good cover. A given local system may, however, be trivialized by many such covers. It is therefore natural to organize these constructions in a cover-independent manner. This should lead to a more intrinsic comparison between connection sheaves, local systems, and flat vector bundles.

Moreover it is known that a differentiable manifold, or a scheme in the algebraic geometry setting, can be recovered from the datum of a set of differentiable algebras satisfying suitable cocycle conditions; see \cite{fioresi2026gluing, salasds, Sancho2017FiniteSpacesSchemes} and the references therein. In \cite{Sancho2017FiniteSpacesSchemes}, an equivalence is established between certain finite ringed spaces and schemes and general properties of finite ringed spaces arising from other objects like manifolds. While in these works the finite ringed spaces considered might not be directly associated with face posets of semisimplicial sets, in \cite{fioresi2026gluing} the authors begin to develop a similar theory for finite ringed spaces arising explicitly from two-dimensional semisimplicial sets. In that setting, the scheme theoretic case is developed in detail, while possible extensions to manifolds are also suggested. It would therefore be desirable to establish a more precise correspondence between these discrete models and manifolds, together with vector bundles and connections on them, thereby extending both the treatment of \emph{op. cit.} and the one developed in this paper. We plan to pursue this direction in future work, including the problem of extending the correspondence between local systems and connection sheaves.
 
On the other hand, arbitrary families of linear edge maps determine left connections by Proposition \ref{prop:edgewise-connections} and open the door to the study of more general geometries described by quiver representations. To extend to a connection sheaf, the edge maps must first be isomorphisms; among families of edge isomorphisms, Proposition \ref{proposition: edge transport extension criterion} identifies curvature on the \(2\)-simplices as precisely the remaining obstruction. We plan to investigate this direction as well, in a way compatible with the finite-ringed-space approach.

\bibliographystyle{dgdict}
\bibliography{bibtex}

\end{document}